\documentclass[12pt]{article}

\usepackage{graphicx}
\usepackage[labelfont=bf,font=sf]{caption}
\usepackage{natbib}
\usepackage{amssymb}
\usepackage{amsmath}
\usepackage{amsthm}
\usepackage{mathtools}
\usepackage[font={sf,small},labelfont=md]{subfig}
\usepackage{transparent}
\usepackage{tikz}
	\usetikzlibrary{trees}
	\tikzstyle{level 1}=[level distance=3cm, sibling distance=3.5cm]
	\tikzstyle{level 2}=[level distance=3cm, sibling distance=2cm]
	\tikzstyle{bag} = [text width=4em, text centered]
	\tikzstyle{end} = [circle, minimum width=3pt,fill, inner sep=0pt]
\usepackage[pscoord]{eso-pic} % The zero point of the coordinate system is the lower left corner of the page (the default).
\usepackage{hyperref}
	\hypersetup{colorlinks=true,linkcolor=black,citecolor=black,urlcolor=black}

\renewcommand{\thefootnote}{\fnsymbol{footnote}}
\newtheorem{theorem}{Theorem}
\newtheorem{conjecture}{Conjecture}
\DeclareMathOperator*{\sgn}{\text{sgn}}
\DeclareMathOperator*{\argmin}{\arg\min}
\DeclareMathOperator*{\shrink}{\text{shrink}}

\newcommand{\given}[2]{\left.{#1}\right|#2}

\newcommand{\bigoh}[1]{\ensuremath{\mathcal{O}\left(#1\right)}}
\newcommand{\Ot}{\ensuremath{\bigoh{t}}}
\newcommand{\prob}[1]{\Pr\left(#1\right)}
\newcommand{\E}[1]{E\left[#1\right]}
\newcommand{\tfun}[2]{\beth\left(#1,#2\right)}
\newcommand{\gammafun}[1]{\Gamma\left(#1\right)}

\title{Inverse Modeling of Dynamical Systems: Multi-Dimensional Extensions of a Stochastic Switching Problem\footnote{This research is supported in part by NSA grant H98230-11-1-0222 and NSF grant DMS-1062817.}}
\author{Erik Bates\footnote{Michigan State University, \protect\href{malito:bateser2@msu.edu}{\protect\nolinkurl{bateser2@msu.edu}}} \and Blake Chamberlain\footnote{Susquehanna University, \protect\href{malito:chamberlainb@susqu.edu}{\protect\nolinkurl{chamberlainb@susqu.edu}}} \and Rachel Gettinger\footnote{Saint Vincent College, \protect\href{malito:rachel.gettinger@email.stvincent.edu}{\protect\nolinkurl{rachel.gettinger@email.stvincent.edu}}}}

\begin{document}
\maketitle
\begin{abstract}
The Buridan's ass paradox is characterized by perpetual indecision between two states, which are never attained. When this problem is formulated as a dynamical system, indecision is modeled by a discrete-state Markov process determined by the system's unknown parameters. Interest lies in estimating these parameters from a limited number of observations. We compare estimation methods and examine how well each can be generalized to multi-dimensional extensions of this system. By quantifying statistics such as mean, variance, frequency, and cumulative power, we construct both method of moments type estimators and likelihood-based estimators. We show, however, why these techniques become intractable in higher dimensions, and thus develop a geometric approach to reveal the parameters underlying the Markov process. We also examine the robustness of this method to the presence of noise. 
\end{abstract}
\newpage
\tableofcontents
\newpage
\renewcommand{\thefootnote}{\arabic{footnote}} \setcounter{footnote}{0}

%%%%%%%%%%%%%%%%%%%%%%%%%%%%%%%%%%%%%%%%%%%%%%%%%%%

\section{Introduction}
In the classical paradox known as ``Buridan's Ass," first posed in the 14th century by Jean Buridan, an ass is placed midway between a pail of water and a bale of hay. Assuming the donkey prefers the closer object, he theoretically never decides to approach one or the other when placed equidistant between the two \citep{lamport}. This problem can be generalized as a dynamical system in which the donkey is free to roam between the two objects, nevertheless changing his preference before ever actually reaching one or the other. In this formulation, the donkey's indecision is modeled by a discrete-state Markov process, the state being the donkey's current preference. The parameters governing this Markov process are the probabilities of the donkey switching that preference. We begin by exploring the mathematical generalization of the classical, two-state paradox with the goal of estimating these parameters. We consider both method of moments type estimators and likelihood-based estimators to achieve this goal. We also develop an algorithm that allows us to directly estimate the parameters.

We are interested in the behavior of the system in higher dimensions, so we also consider the donkey in a triangular pen. By adding one more state to the system and extending the problem from one to two dimensions, many of the methods we explore for the donkey on a line become intractable. Nevertheless, our algorithm, which detects the donkey's state at each unit time step, still allows us to measure the parameters of the system. In the absence of measurement noise, this state detector gives us the most reasonable values for our parameters. In the presence of noise, though, the state detector performs poorly. Consequently, we employ methods of denoising the system before applying state detection.

%%%%%%%%%%%%%%%%%%%%%%%%%%%%%%%%%%%%%%%%%%%%%%%%%%%

\section{Donkey on a Line}
Following the setup of Buridan's Ass, we place the donkey on a line between a pail of water at $x = 0$ and a bale of hay at $x = 1$.  We define state 0 as the state in which the donkey is moving towards the pail of water, and state 1 as the state in which the donkey is moving towards the bale of hay.\footnote{For simplicity, the donkey is always initially placed at $x=0.5$ and is taken to be in state 0.} The donkey slows as he comes closer to either object, so he never actually reaches the water or the hay.  The donkey's movement is modeled by the following differential equations:
\begin{align}
\text{State 0: } \dfrac{dx}{dt} &= v(-x) \label{eq: state0} \\
\text{State 1: } \dfrac{dx}{dt} &= v(1-x). \label{eq: state1}
\end{align}
The constant $v$ controls the speed of the donkey. While the donkey obeys these differential equations in a continuous fashion, we allow him to change state only at unit time steps. We let $\tau_{01}$ be the probability that the donkey switches from state 0 to state 1 at any given time step, and $\tau_{10}$ be the probability that the donkey switches from state 1 to state 0. These are the parameters we estimate in our mathematical model of Buridan's Ass. We assume we know the constant $v$ (even if not known, it is easily measured), and we observe only the donkey's position at each of finitely many time steps. Notice that the parameters $\tau_{01}$ and $\tau_{10}$ change with neither time nor the donkey's current state, imparting them the independence that characterizes the stochastic switching of this system as a Markov process.

%%%%%%%%%%%%%%%%%%%%%%%%%%%%%%%%%%%%%%%%%%%%%%%%%%%

\subsection{Method of Moments Type Estimators}
One way to estimate $\tau_{01}$ and $\tau_{10}$ is the method of moments approach.  The standard method of moments estimation scheme uses moments of measured data, such as mean and variance, to estimate unknown parameters \citep{moments}. The approach relies on deriving algebraic expressions for these moments in terms of the parameters. When these expressions are invertible, they give closed-form estimates of parameters in terms of only measured data. Furthermore, this technique can be used with feature statistics such as frequency and cumulative power. Since the donkey on a line system has two parameters, we require couplets, or pairs of invertible expressions, to estimate $\tau_{01}$ and $\tau_{10}.$ 

%%%%%%%%%%%%%%%%%%%%%%%%%%%%%%%%%%%%%%%%%%%%%%%%%%%

\subsubsection{The Markov Process} \label{Donkey on a Line: The Markov Process}
The probability distribution of the donkey's state at the $n$th time step is given by the discrete-state Markov chain
\begin{align}
p_{n}&=Ap_{n-1}, \label{eq: markov}
\end{align}
where 
\begin{align}
A=\begin{bmatrix}
1-\tau_{01}&\tau_{10} \\
\tau_{01}&1-\tau_{10}
\end{bmatrix}
\end{align}
and $p_0$ is a vector whose entries are the initial probabilities of the donkey being in either state. Notice that $A$ is a column stochastic matrix (one whose columns each have entries summing to 1). Any such matrix has an eigenvector associated with an eigenvalue of 1 (see \hyperref[app: eigenvector]{Appendix \ref*{app: eigenvector}}). It is a well-known fact that under certain conditions\footnote{The Markov chain must be recurrent, aperiodic, and irreducible \citep{ergodic}. A recurrent Markov chain is one in which any state, once reached, will with probability 1 be reached again. Aperiodicity is the condition that for each state $i$, $\gcd\{n: \text{state }i \text{ can be returned to in } n \text{ steps}\} = 1.$ Finally, an irreducible chain is one in which any state can somehow be reached from any other state.}, this eigenvector is unique\footnote{Up to multiplication by a constant. For the eigenvector $\mathbf{v}$ to make sense as a probability distribution, it must be normalized so that its entries sum to 1.} and represents a stable distribution to which any initial distribution vector will converge after repeated applications of \eqref{eq: markov}. For $A,$ this eigenvector is
\begin{align}
\mathbf{v} = \begin{bmatrix}
\dfrac{\tau_{10}}{\tau_{01}+\tau_{10}} \\[0.4cm] 
\dfrac{\tau_{01}}{\tau_{01}+\tau_{10}}
\end{bmatrix}.
\end{align}
This vector shows that the expected proportion of time the donkey spends in state  0 is $\frac{\tau_{10}}{\tau_{01}+\tau_{10}}$, and the expected proportion of time he spends in state 1 is $\frac{\tau_{01}}{\tau_{01}+\tau_{10}}.$  These proportions motivate the conditional probability tree in \autoref{fig: tree}.
\begin{figure}[h]
\centering
\begin{tikzpicture}[grow=right]
\node[bag] {Donkey}
    child {
        node[bag] {State $1$}        
            child {
                node[end, label=right:
                    {State 1 \ $\displaystyle\frac{\left(1-\tau_{10}\right)\tau_{01}}{\tau_{01}+\tau_{10}}$}] {}
                edge from parent
                node[above] {}
                node[below]  {$\phantom{\frac{0}{0}}1-\tau_{10}\phantom{000}$}
            }
            child {
                node[end, label=right:
                    {State 0 \ $\displaystyle\frac{\tau_{01}\tau_{10}}{\tau_{01}+\tau_{10}}$}] {}
                edge from parent
                node[above] {$\tau_{10}$}
                node[below]  {}
            }
            edge from parent 
            node[above] {}
            node[below]  {$\displaystyle\frac{\tau_{01}}{\tau_{01}+\tau_{10}}\phantom{00w000}$}
    }
    child {
        node[bag] {State $0$}        
        child {
                node[end, label=right:
                    {State 1 \ $\displaystyle\frac{\tau_{01}\tau_{10}}{\tau_{01}+\tau_{10}}$}] {}
                edge from parent
                node[above] {}
                node[below]  {$\tau_{01}$}
            }
            child {
                node[end, label=right:
                    {State 0 \ $\displaystyle\frac{\left(1-\tau_{01}\right)\tau_{10}}{\tau_{01}+\tau_{10}}$}] {}
                edge from parent
                node[above] {$1-\tau_{01} \phantom{\frac{01}{0}}$}
                node[below]  {}
            }
        edge from parent         
            node[above] {$\displaystyle\frac{\tau_{10}}{\tau_{01}+\tau_{10}}\phantom{\frac{1}{100000}}$}
            node[below]  {}
    };
\end{tikzpicture}
\caption{Conditional probability tree for donkey on a line}
\label{fig: tree}
\end{figure}
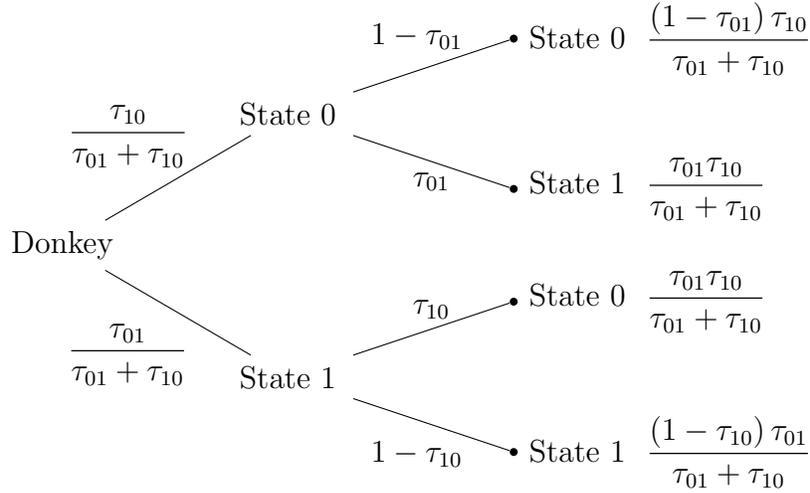

The first two branches of the tree show the marginal probabilities of the donkey being in either state, as given by the eigenvector $\mathbf{v}.$ The second layer of branches identifies the conditional probabilities of the donkey staying in his current state or transitioning to the other. These probabilities come simply from the problem's construction. By the independence of the Markov chain, we can multiply connected branches to produce the final, joint probabilities of the tree. For example, the probability of the donkey being in state 0 and then switching to state 1 is $\frac{\tau_{01}\tau_{10}}{\tau_{01}+\tau_{10}}$.  Notice this is equal to the probability of the donkey being in state 1 and then switching to state 0.  Therefore, we define our frequency of transition as 
\begin{align}
\omega = \frac{\tau_{01}\tau_{10}}{\tau_{01}+\tau_{10}}. \label{eq: frequency}
\end{align} 
We may think of $\omega$ as the probability of observing some state transition at a given time step. This is our first statistic for use in a method of moments couplet.

%%%%%%%%%%%%%%%%%%%%%%%%%%%%%%%%%%%%%%%%%%%%%%%%%%%

\subsubsection{Continuous Dynamics}
Although the donkey's state is discrete, his position is continuous in both space and time.  We let $P(x,t)$ be the probability density of the donkey's position, $x$, at time $t.$ Note that for fixed $t$, $P(x,t)$ is a probability density of one variable. Furthermore, since the donkey cannot simultaneously be in both states, we can express $P$ as the sum of state-dependent, conditional probabilities,\footnote{By referring to $P_0$ and $P_1$ as probabilities, we do not mean that their integrals over their support equal 1, but rather that they sum to a valid probability density function. If properly normalized, $P_0$ and $P_1$ would be the probability densities of the donkey's position given his being in state 0 or state 1, respectively.} $P_0(x,t)$ and $P_1(x,t)$:
\begin{align}
P(x,t) = P_0(x,t) + P_1(x,t).
\end{align}
Over time, the probability associated with state 0 is advected towards 0 and the probability associated with state 1 is advected towards 1. Probability is also transferred among states by the switching probabilities, $\tau_{01}$ and $\tau_{10}$. These changes in probability density are captured by the conservation conditions
\begin{align}
\frac{\partial{P_{0}}}{\partial{t}} &= -\frac{\partial}{\partial{x}}\left[v(-x)P_{0}(x,t)\right] - \tau_{01}{P}_{0}(x,t) + \tau_{10}{P}_{1}(x,t) \label{eq: cc1}\\
\frac{\partial{P_{1}}}{\partial{t}} &= -\frac{\partial}{\partial{x}}\left[v\left(1-x\right)P_{1}(x,t)\right] + \tau_{01}{P}_{0}(x,t) - \tau_{10}{P}_{1}(x,t). \label{eq: cc2}
\end{align}

We are interested in the long term behavior of the system, and thus we look for time-invariant steady-state solutions. That is, we examine the case in which
\begin{align}
\frac{\partial{P_{0}}}{\partial{t}} = \frac{\partial{P_{1}}}{\partial{t}} = 0,
\end{align}
and we can, therefore, drop the dependence of $P_0$ and $P_1$ on $t.$ Furthermore, we anticipate that in a steady-state solution, total fluxes are balanced:
\begin{align}
v(-x)P_0(x) &= v(1-x)P_1(x) \\
P_0(x) &= \frac{1-x}{x}P_1(x). \label{eq: balance}
\end{align}
The steady-state condition in state 1 is now
\begin{align}
0 &= 	\dfrac{\partial{P_{1}}}{\partial{t}} \\
&= -\dfrac{\partial}{\partial{x}}\left[v\left(1-x\right)P_{1}(x)\right] + \tau_{01}\frac{1-x}{x}P_{1}(x) - \tau_{10}{P}_{1}(x). \label{eq: prederiv}
\end{align}
Applying the derivative operator in \eqref{eq: prederiv} and rearranging terms, we have
\begin{align}
 -v\left(1-x\right)\dfrac{\partial{P_{1}}}{\partial{x}} + vP_{1}(x) &=-\tau_{01}\frac{1-x}{x}P_{1}(x) + \tau_{10}{P}_{1}(x)\\
	 v\left(1-x\right)\dfrac{\partial{P_{1}}}{\partial{x}} &=\tau_{01}\frac{1-x}{x}P_{1}(x) - \tau_{10}{P}_{1}(x) + vP_{1}(x)\\
	 \dfrac{\partial{P_{1}}}{\partial{x}} &=\tau_{01}\frac{1}{vx}P_{1}(x) - \tau_{10}\frac{1}{v\left(1-x\right)}P_{1}(x) + \frac{1}{1-x}P_{1}(x).
\end{align}
This is a separable first-order ordinary differential equation.  Using separation of variables we get 
\begin{align}
\frac{1}{P_{1}(x)}\partial{P_{1}} &= \left[\frac{\tau_{01}}{v}\frac{1}{x} - \frac{\tau_{10}}{v}\frac{1}{1-x} + \frac{1}{\left(1-x\right)}\right]\partial{x} \\
\ln\left(P_{1}(x)\right) &= \frac{\tau_{01}}{v}\ln\left(x\right) + \frac{\tau_{10}}{v}\ln\left(1-x\right) - \ln\left(1-x\right) + K \\
P_{1}(x) &= Cx^{\frac{\tau_{01}}{v}}\left(1-x\right)^{\frac{\tau_{10}}{v}-1}.
\end{align}
Substituting $P_1(x)$ into \eqref{eq: balance}, we find that
\begin{align}
P_0(x) = Cx^{\frac{\tau_{01}}{v}-1}\left(1-x\right)^{\frac{\tau_{10}}{v}}.
\end{align}
We then combine $P_0$ and $P_1$, yielding the probability distribution
\begin{align}
P(x) &= P_{0}(x) + P_{1}(x) \\
&= Cx^{\frac{\tau_{01}}{v}-1}\left(1-x\right)^{\frac{\tau_{10}}{v}-1}. \label{eq: beta}
\end{align}
Thus, the constant
\begin{align}
C = \frac{1}{\tfun{\frac{\tau_{01}}{v}}{\frac{\tau_{10}}{v}}},
\end{align}
where
\begin{align}
\tfun{\frac{\tau_{01}}{v}}{\frac{\tau_{10}}{v}} = {\int_0^1 x^{\frac{\tau_{01}}{v}-1}(1-x)^{\frac{\tau_{10}}{v}-1} \ dx}. 
\end{align}
This makes $P(x)$ a beta probability distribution \citep{johnson} with parameters $\frac{\tau_{01}}{v}$ and $\frac{\tau_{10}}{v}.$  The first two moments of $P(x)$ are thus
\begin{align}
\mu&=\frac{\tau_{01}}{\tau_{01}+\tau_{10}} \label{eq: mean}
\end{align}
and
\begin{align}
	\sigma^2&=\frac{\tau_{01}\tau_{10}}{\left(\tau_{01}+\tau_{10}\right)^2\left(\frac{\tau_{01}}{v}+\frac{\tau_{10}}{v}+1\right)}. \label{eq: variance}
\end{align}
The mean, $\mu,$ and the variance, $\sigma^2,$ provide two more statistics with which to form couplets.

%%%%%%%%%%%%%%%%%%%%%%%%%%%%%%%%%%%%%%%%%%%%%%%%%%%

\subsubsection{Cumulative Power}
A final statistic we examine is cumulative power. For a twice differentiable signal $m,$ and finite $t$, cumulative power, $F$, is defined as
\begin{align}
F(t) = \int_0^t \bigl(m''(s)\bigr)^2\ ds. \label{eq: cumpow}
\end{align}
Cumulative power has been shown to be $\mathcal{O}(t)$ for any $m$ that is a finite sum of sine and cosine functions. Moreover, the average derivative of $F$ for such an $m$ is pivotal with respect to the frequency and the amplitude of $m$ \citep{keaton}. Thus, measurements of the rate of growth in $F$ (with respect to time) can provide statistics amenable to method of moments type estimators.

In the donkey on a line system, we take $m$ to be the donkey's position, $x$. Although $x$ is not a periodic function, we derive a similar condition to the above for its cumulative power. First, however, we must adjust the definition \eqref{eq: cumpow} to account for the times $t_1,\dots,t_n$ at which $x$ is not twice differentiable. We note that $t_1,\dots,t_n$ are the times at which the donkey changes state. From here on, we redefine the cumulative power of $x$ as
\begin{align}
F(t) = \int_0^{t_1} \bigl(x''(s)\bigr)^2\ ds + \int_{t_1}^{t_2} \bigl(x''(s)\bigr)^2\ ds + \cdots + \int_{t_{n-1}}^{t_n} \bigl(x''(s)\bigr)^2\ ds + \int_{t_{n}}^{t} \bigl(x''(s)\bigr)^2\ ds, \label{eq: cumpow2}
\end{align}
where $x'' = \frac{d^2x}{dt^2}.$ In \hyperref[app: power]{Appendix \ref*{app: power}}, $F$ is shown to be $\Ot$. Anticipating, then, that $F(t)$ can be approximated by a line, we look to obtain another statistic for use in the method of moments approach: the expected value of $\frac{dF}{dt},$ in terms of $\tau_{01}$ and $\tau_{10}.$ In carrying out the calculation, we employ the gamma function $\Gamma: \mathbb{R}^+ \rightarrow \mathbb{R}$ defined\footnote{$\mathbb{R}^+ $denotes the set of positive real numbers. The gamma function can also be defined for complex numbers with positive real part \citep{formulas}.} by 
\begin{align}
\gammafun{a} = \int_0^\infty e^{-t} t^{a-1} dt.
\end{align}
Two useful properties of the gamma function \citep{formulas} are
\begin{align}
\gammafun{a}\gammafun{b} = \tfun{a}{b}\gammafun{a+b} \label{eq: prop1} \\
\gammafun{a+1} = a\gammafun{a}. \label{eq: prop0}
\end{align}
From \eqref{eq: prop0}, we easily obtain the identity
\begin{align}
\gammafun{a+2} &= \left(a^{2}+a\right)\gammafun{a}. \label{eq: prop2}
\end{align}

Now beginning our derivation of the expected value of $\frac{dF}{dt}$, in state 0 we have 
\begin{align}
(x'')^2&=\bigl[(-vx)'\bigr]^{2}=(v^{2}x)^{2}=v^{4}x^{2}, \label{eq: power1}
\end{align}
and in state 1 we have
\begin{align}
(x'')^2&=\Bigl[\bigl(v(1-x)\bigr)'\Bigr]^{2}=\bigl(v^{2}(x-1)\bigr)^{2}=v^{4}(1-x)^{2}. \label{eq: power2}
\end{align}
Let $z$ be the donkey's current state. From \eqref{eq: power1} and \eqref{eq: power2}, it is clear that $\frac{dF}{dt}$ is conditional upon the donkey's current position and state. The expectation for $\frac{dF}{dt}$, therefore, is found by averaging $\frac{dF}{dt}$ over all possible position and state combinations. The likelihood of observing any one of these combinations is conditional upon the parameters $\tau_{01}$ and $\tau_{10}.$ More formally,
\begin{align}
\begin{split} 
\E{\given{\frac{dF}{dt}}{\tau_{01},\tau_{10}}} &=
\int_{0}^{1}\big[v^{4}x^{2}\prob{x | z=0,\tau_{01},\tau_{10}}\prob{z=0 | \tau_{01},\tau_{10}}\\
&\phantom{=\int_0^1\big[}+v^{4}\left(1-x\right)^{2}\prob{x | z=1,\tau_{01},\tau_{10}}\prob{z = 1 | \tau_{01},\tau_{10}}\big] dx. \label{eq: int1}
\end{split}
\end{align}
 
For ease of notation, we denote $S = \E{\given{\frac{dF}{dt}}{\tau_{01},\tau_{10}}}.$ Substituting known probabilities and probability densities, \eqref{eq: int1} becomes
\begin{align}
\begin{split}
S &=v^{4}\Bigg[\int_{0}^{1}x^{2}\frac{1}{\tfun{\frac{\tau_{01}}{v}}{\frac{\tau_{10}}{v}+1}}x^{\frac{\tau_{01}}{v}-1}(1-x)^{\frac{\tau_{10}}{v}}\left(\frac{\tau_{10}}{\tau_{01} + \tau_{10}}\right)dx\\
&\phantom{=v^{4}\Bigg[} + \int_{0}^{1} \left(1-x\right)^{2}\frac{1}{\tfun{\frac{\tau_{01}}{v}+1}{\frac{\tau_{10}}{v}}}x^{\frac{\tau_{01}}{v}}(1-x)^{\frac{\tau_{10}}{v}-1}\left(\frac{\tau_{01}}{\tau_{01} + \tau_{10}}\right)dx\Bigg].
\end{split}
\end{align}
We can multiply by fractions equal to 1:
\begin{align}
\begin{split}
S&=v^{4}\Bigg[\int_{0}^{1}\frac{1}{\tfun{\frac{\tau_{01}}{v}}{\frac{\tau_{10}}{v}+1}}\frac{\tfun{\frac{\tau_{01}}{v}+2}{\frac{\tau_{10}}{v}+1}}{\tfun{\frac{\tau_{01}}{v}+2}{\frac{\tau_{10}}{v}+1}}x^{\frac{\tau_{01}}{v}+1}(1-x)^{\frac{\tau_{10}}{v}}\left(\frac{\tau_{01}}{\tau_{01} + \tau_{10}}\right)dx\\
&\phantom{=v^{4}\Bigg[}+ \int_{0}^{1}\frac{1}{\tfun{\frac{\tau_{01}}{v}+1}{\frac{\tau_{10}}{v}}}\frac{\tfun{\frac{\tau_{01}}{v}+1}{\frac{\tau_{10}}{v}+2}}{\tfun{\frac{\tau_{01}}{v}+1}{\frac{\tau_{10}}{v}+2}}x^{\frac{\tau_{01}}{v}}(1-x)^{\frac{\tau_{10}}{v}+1}\left(\frac{\tau_{01}}{\tau_{01} + \tau_{10}}\right)dx\Bigg].
\end{split} \label{eq: ugly1}
\end{align}
Factoring out certain constants in \eqref{eq: ugly1} we get
\begin{align}
\begin{split}
S &= v^{4}\Bigg[\frac{\tfun{\frac{\tau_{01}}{v}+2}{\frac{\tau_{10}}{v}+1}}{\tfun{\frac{\tau_{01}}{v}}{\frac{\tau_{10}}{v}+1}}\left(\frac{\tau_{10}}{\tau_{01} + \tau_{10}}\right)\int_{0}^{1}\frac{1}{\tfun{\frac{\tau_{01}}{v}+2}{\frac{\tau_{10}}{v}+1}}x^{\frac{\tau_{01}}{v}+1}(1-x)^{\frac{\tau_{10}}{v}}dx\\
&\phantom{=}+ \frac{\tfun{\frac{\tau_{01}}{v}+1}{\frac{\tau_{10}}{v}+2}}{\tfun{\frac{\tau_{01}}{v}+1}{\frac{\tau_{10}}{v}}}\left(\frac{\tau_{01}}{\tau_{01} + \tau_{10}}\right)\int_{0}^{1}\frac{1}{\tfun{\frac{\tau_{01}}{v}+1}{\frac{\tau_{10}}{v}+2}}x^{\frac{\tau_{01}}{v}}(1-x)^{\frac{\tau_{10}}{v}+1}dx\Bigg].
\end{split} \label{eq: ugly2}
\end{align}
Recognizing that the integrals in \eqref{eq: ugly2} are of densities taken over their support and are thus equal to
1, we are left with
\begin{align}
S = v^{4}\Bigg[\frac{\tfun{\frac{\tau_{01}}{v}+2}{\frac{\tau_{10}}{v}+1}}{\tfun{\frac{\tau_{01}}{v}}{\frac{\tau_{10}}{v}+1}}\left(\frac{\tau_{10}}{\tau_{01} + \tau_{10}}\right) + \frac{\tfun{\frac{\tau_{01}}{v}+1}{\frac{\tau_{10}}{v}+2}}{\tfun{\frac{\tau_{01}}{v}+1}{\frac{\tau_{10}}{v}}}\left(\frac{\tau_{01}}{\tau_{01} + \tau_{10}}\right)\Bigg]. \label{eq: ugly3}
\end{align}
Applying \eqref{eq: prop1} to the first term of the right-hand side of \eqref{eq: ugly3} yields
		\begin{align}
		\frac{\beth\left(\frac{\tau_{01}}{v}+2,\frac{\tau_{10}}{v}+1\right)}{\beth\left(\frac{\tau_{01}}{v},\frac{\tau_{10}}{v}+1\right)}\left(\frac{\tau_{10}}{\tau_{01}+\tau_{10}}\right)&=\frac{\Gamma\left(\frac{\tau_{01}}{v}+2\right)\Gamma\left(\frac{\tau_{10}}{v}+1\right)\frac{1}{\Gamma\left(\frac{\tau_{01}}{v}+\frac{\tau_{10}}{v}+3\right)}}{\Gamma\left(\frac{\tau_{01}}{v}\right)\Gamma\left(\frac{\tau_{10}}{v}+1\right)\frac{1}{\Gamma\left(\frac{\tau_{01}}{v}+\frac{\tau_{10}}{v}+1\right)}}\left(\frac{\tau_{10}}{\tau_{01}+\tau_{10}}\right) \\
&=\frac{\Gamma\left(\frac{\tau_{01}}{v}+2\right)\Gamma\left(\frac{\tau_{01}}{v}+\frac{\tau_{10}}{v}+1\right)}{\Gamma\left(\frac{\tau_{01}}{v}\right)\Gamma\left(\frac{\tau_{01}}{v}+\frac{\tau_{10}}{v}+3\right)}\left(\frac{\tau_{10}}{\tau_{01}+\tau_{10}}\right). \label{eq: gammabunch}
		\end{align}
		Now, by applying \eqref{eq: prop2} to \eqref{eq: gammabunch} we have
		\begin{align}
		\frac{\beth\left(\frac{\tau_{01}}{v}+2,\frac{\tau_{10}}{v}+1\right)}{\beth\left(\frac{\tau_{01}}{v},\frac{\tau_{10}}{v}+1\right)}\left(\frac{\tau_{10}}{\tau_{01}+\tau_{10}}\right)&=\frac{\left(\frac{\tau_{01}}{v}\right)^{2}+\frac{\tau_{01}}{v}}{\left(\frac{\tau_{01}}{v}+\frac{\tau_{10}}{v}+1\right)^{2}+\left(\frac{\tau_{01}}{v}+\frac{\tau_{10}}{v}+1\right)}\left(\frac{\tau_{10}}{\tau_{01}+\tau_{10}}\right). \label{eq: term1}
		\end{align}
		A similar computation on the second term of the right-hand side of \eqref{eq: ugly3} yields
		\begin{align}
		\frac{\beth\left(\frac{\tau_{01}}{v}+1,\frac{\tau_{10}}{v}+2\right)}{\beth\left(\frac{\tau_{01}}{v}+1,\frac{\tau_{10}}{v}\right)}\left(\frac{\tau_{01}}{\tau_{01}+\tau_{10}}\right)&=\frac{\left(\frac{\tau_{10}}{v}\right)^{2}+\frac{\tau_{10}}{v}}{\left(\frac{\tau_{01}}{v}+\frac{\tau_{10}}{v}+1\right)^{2}+\left(\frac{\tau_{01}}{v}+\frac{\tau_{10}}{v}+1\right)}\left(\frac{\tau_{01}}{\tau_{01}+\tau_{10}}\right). \label{eq: term2}
		\end{align}
		Now, substituting \eqref{eq: term1} and \eqref{eq: term2} into \eqref{eq: ugly3}, we have
		\begin{align}
		S &=v^{4}\Bigg[\frac{\left(\left(\frac{\tau_{01}}{v}\right)^{2}+\frac{\tau_{01}}{v}\right)\tau_{10}+\left(\left(\frac{\tau_{10}}{v}\right)^{2}+\frac{\tau_{10}}{v}\right)\tau_{01}}{(\tau_{01}+\tau_{10})\left[\left(\frac{\tau_{01}}{v}+\frac{\tau_{10}}{v}+1\right)^2+\left(\frac{\tau_{01}}{v}+\frac{\tau_{10}}{v}+1\right)\right]}\Bigg]. \label{eq: final1}
		\end{align}
		Finally, simplifying \eqref{eq: final1} gives
		\begin{align}
		S &=v^{4}\Bigg[\frac{\frac{\tau_{01}\tau_{10}}{v}\left(\frac{\tau_{01}}{v}+\frac{\tau_{10}}{v}+2\right)}{(\tau_{01}+\tau_{10})\left(\frac{\tau_{01}}{v}+\frac{\tau_{10}}{v}+1\right)\left(\frac{\tau_{01}}{v}+\frac{\tau_{10}}{v}+2\right)}\Bigg] \\
		&=\frac{v^{4}\tau_{01}\tau_{10}}{(\tau_{01}+\tau_{10})(\tau_{01}+\tau_{10}+v)}.
		\end{align}
Therefore, the average slope of the cumulative power of $x$ is 
\begin{align}
S=\frac{v^{4}\tau_{01}\tau_{10}}{(\tau_{01}+\tau_{10})(\tau_{01}+\tau_{10}+v)}, \label{eq: slope1}
\end{align}
giving us our final statistic for use in a couplet.

%%%%%%%%%%%%%%%%%%%%%%%%%%%%%%%%%%%%%%%%%%%%%%%%%%%

\subsubsection{Couplet Inversion and Results}
Now we have derived closed-form expressions for mean, variance, frequency, and the average slope of cumulative power, in terms of $\tau_{01}$ and $\tau_{10}.$ The method of moments approach demands, however, that each of these statistics can be measured given the data. For mean and variance, this is clearly possible. Measuring frequency, though, requires accurate detection of state transitions. Finally, cumulative power is perhaps the most difficult to measure, as it requires knowledge of the donkey's state at all time steps. In the following simulations, we ignore these issues as to identify which couplets perform most effectively. That is, we assume we have perfect measurements of the four chosen statistics.

To complete our method of moments approach, we find invertible couplets of our closed-form expressions. Examining the equation for each statistic (\eqref{eq: frequency}, \eqref{eq: mean}, \eqref{eq: variance}, and \eqref{eq: slope1}), we find that each one is symmetric with respect to $\tau_{01}$ and $\tau_{10},$ with the exception of the mean, $\mu.$  This implies that any invertible couplet will include the mean. For example, inverting the equations for mean, $\mu=\frac{\tau_{01}}{\tau_{01}+\tau_{10}}$, and frequency, $\omega=\frac{\tau_{01}\tau_{10}}{\tau_{01}+\tau_{10}}$, we find
\begin{subequations} \label{eq: MF}
\begin{align}
	\tau_{01}&=\frac{\omega}{1-\mu} \label{eq: MFa} \\
	\tau_{10}&=\frac{\omega}{\mu}. \label{eq: MFb}
\end{align}
\end{subequations}
Inverting the equations for mean and variance, $\sigma^2=\frac{\tau_{01}\tau_{10}}{\left(\tau_{01}+\tau_{10}\right)^2\left(\frac{\tau_{01}}{v}+\frac{\tau_{10}}{v}+1\right)}$, we obtain
\begin{subequations} \label{eq: MV}
\begin{align}
	\tau_{01}&=\frac{\mu^{2}v\left(1-\mu\right)}{\sigma^2}-\mu v\\
	\tau_{10}&=\frac{v(\mu-1)\left(\mu^2-\mu+\sigma^2\right)}{\sigma^2}.	
\end{align}
\end{subequations}
Finally, inverting the equations for mean and average cumulative power slope,\\ $S=\frac{v^{4}\tau_{01}\tau_{10}}{(\tau_{01}+\tau_{10})(\tau_{01}+\tau_{10}+v)}$, we find
\begin{subequations} \label{eq: MS}
\begin{align}
	\tau_{01}&=\frac{\mu Sv}{v^4\left(\mu-\mu^2-S\right)}\\
	\tau_{10}&=\frac{Sv\left(\mu-1\right)}{\mu^2v^4-\mu v^4+S}.	
\end{align}
\end{subequations}

Each of these couplet inversions vary in accuracy and precision, depending on the true values of $\tau_{01}$ and $\tau_{10}.$ For instance, \autoref{fig: surfaces} shows how well the mean and frequency couplet \eqref{eq: MF} performs for different pairs of $\tau_{01}$ and $\tau_{10}.$ The surfaces shown were generated by running twenty simulations of donkey on a line, each one for 20,000 time steps, for every pair of parameters. There were 361 total pairs, created by meshing grids of 19 equally spaced points from 0.01 to 0.1. In each diagram, the vertical axis identifies the average absolute error seen in the estimates \eqref{eq: MFa} and \eqref{eq: MFb}. We can see that estimates improve as $\tau_{01}$ and $\tau_{10}$ both approach zero. Similar behavior was observed for the couplets \eqref{eq: MV} and \eqref{eq: MS}.

\begin{figure}[h]
\subfloat[Error in \eqref{eq: MFa}]{
\includegraphics[width=0.47\textwidth]{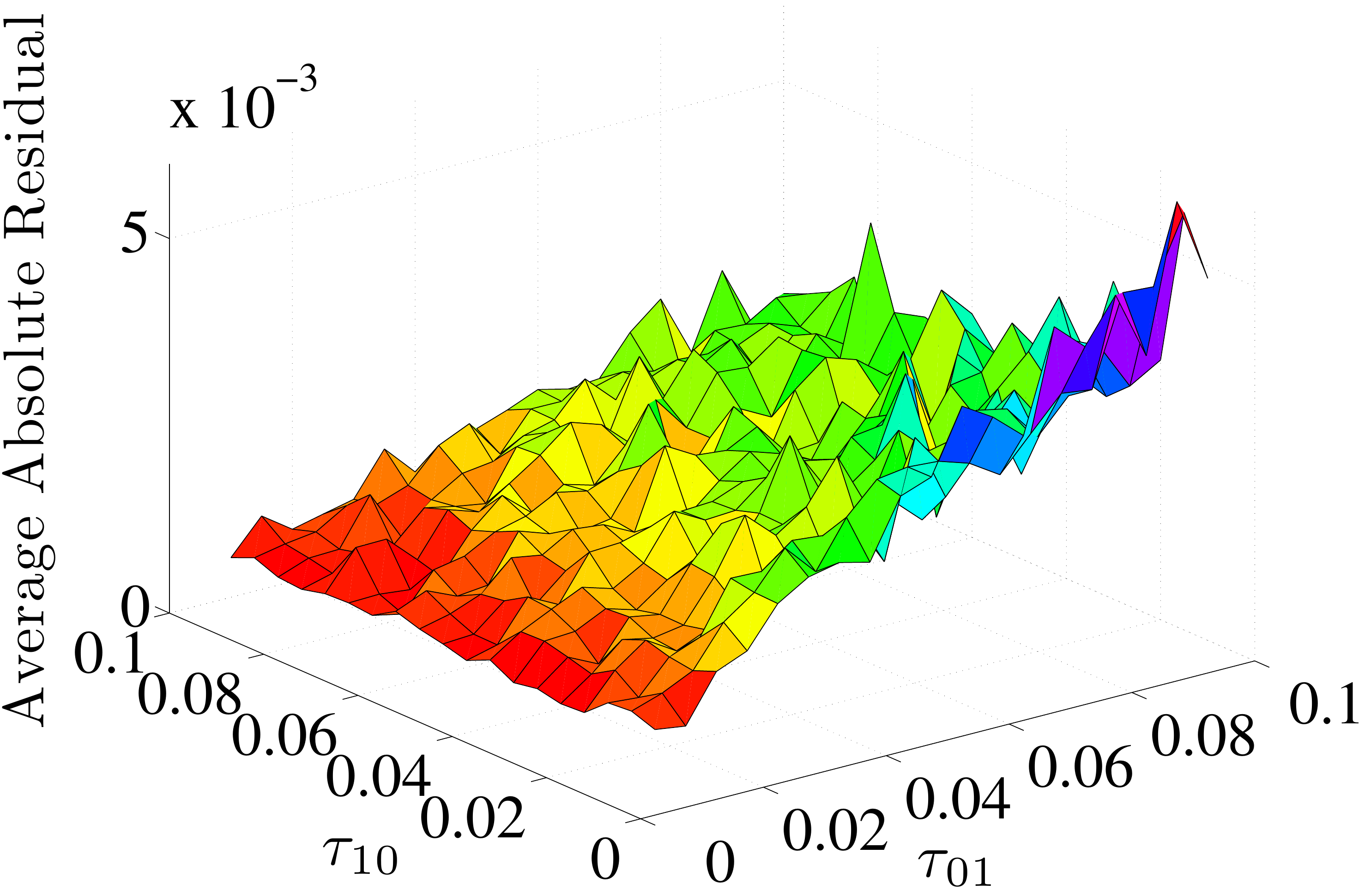}
}
\hfill
\subfloat[Error in \eqref{eq: MFb}]{
\includegraphics[width=0.47\textwidth]{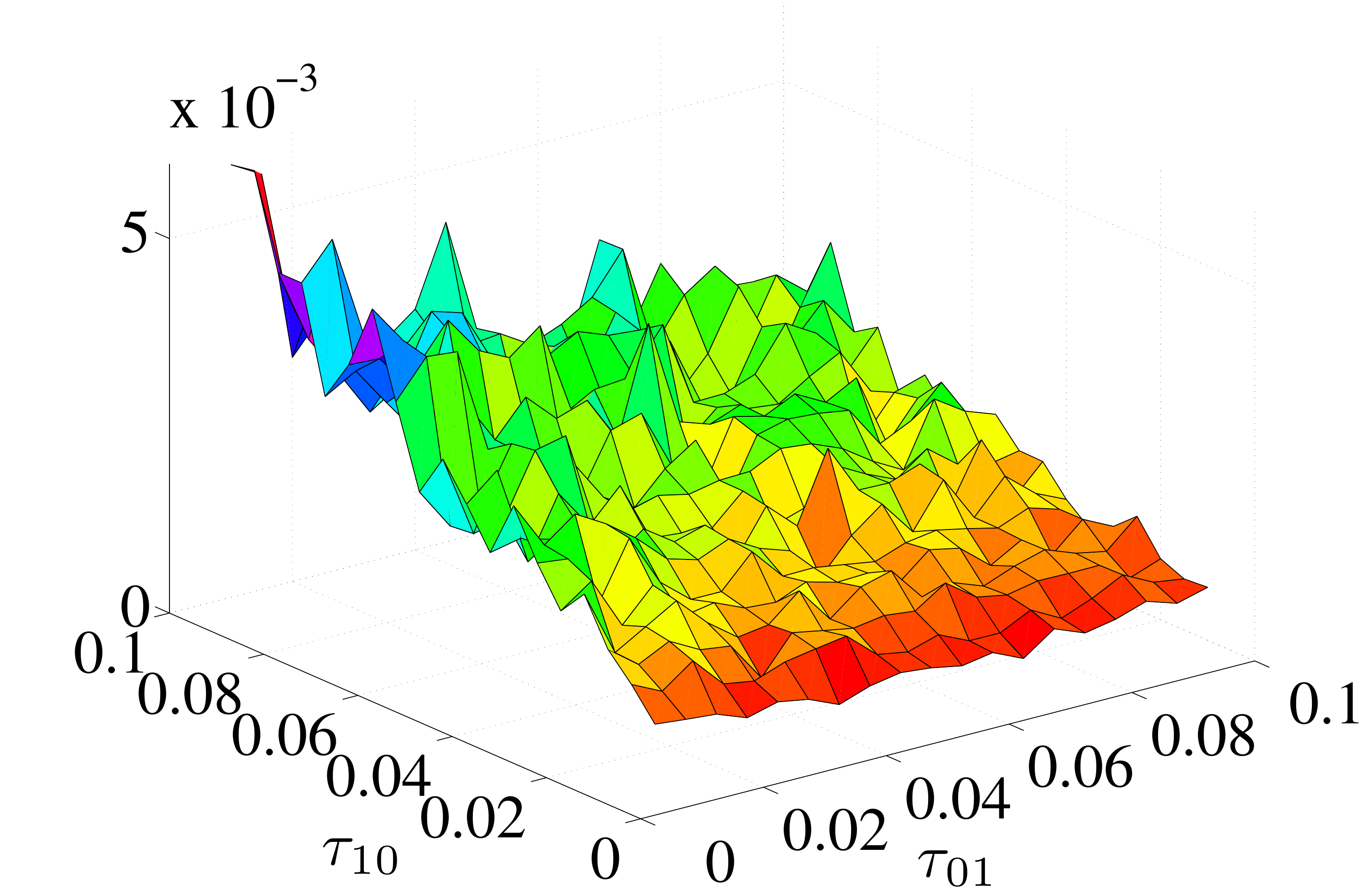}
}
\caption{Average absolute residuals in parameter estimations using the mean and frequency couplet \eqref{eq: MF}}
\label{fig: surfaces}
\end{figure}

\begin{figure}[!ht]
\subfloat[Mean and Variance, $\tau_{01}$]{
\includegraphics[width=0.47\textwidth,clip=true,trim=0 2cm 0 0]{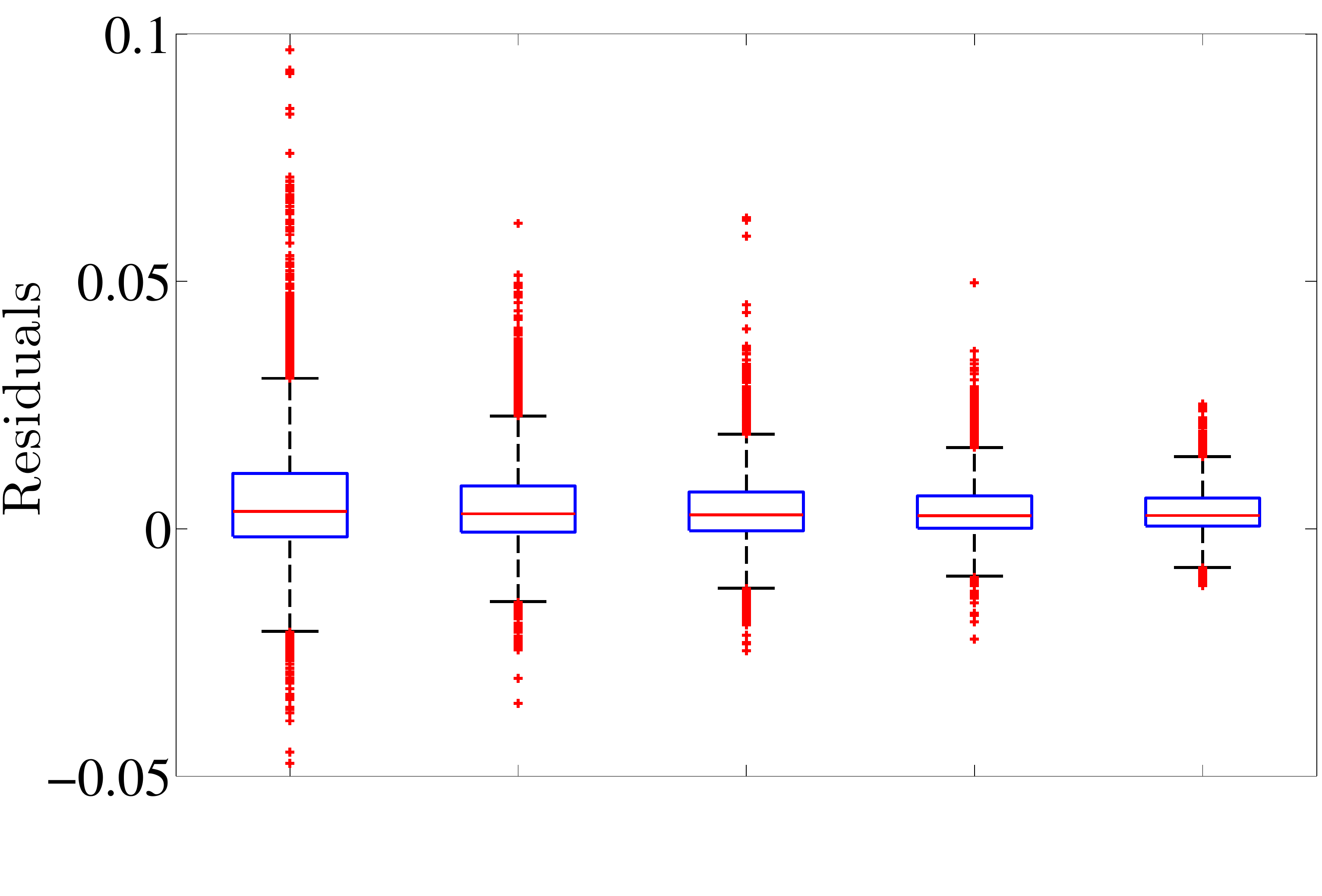}
}
\hfill
\subfloat[Mean and Variance, $\tau_{10}$]{
\includegraphics[width=0.47\textwidth,clip=true,trim=0 2cm 0 0]{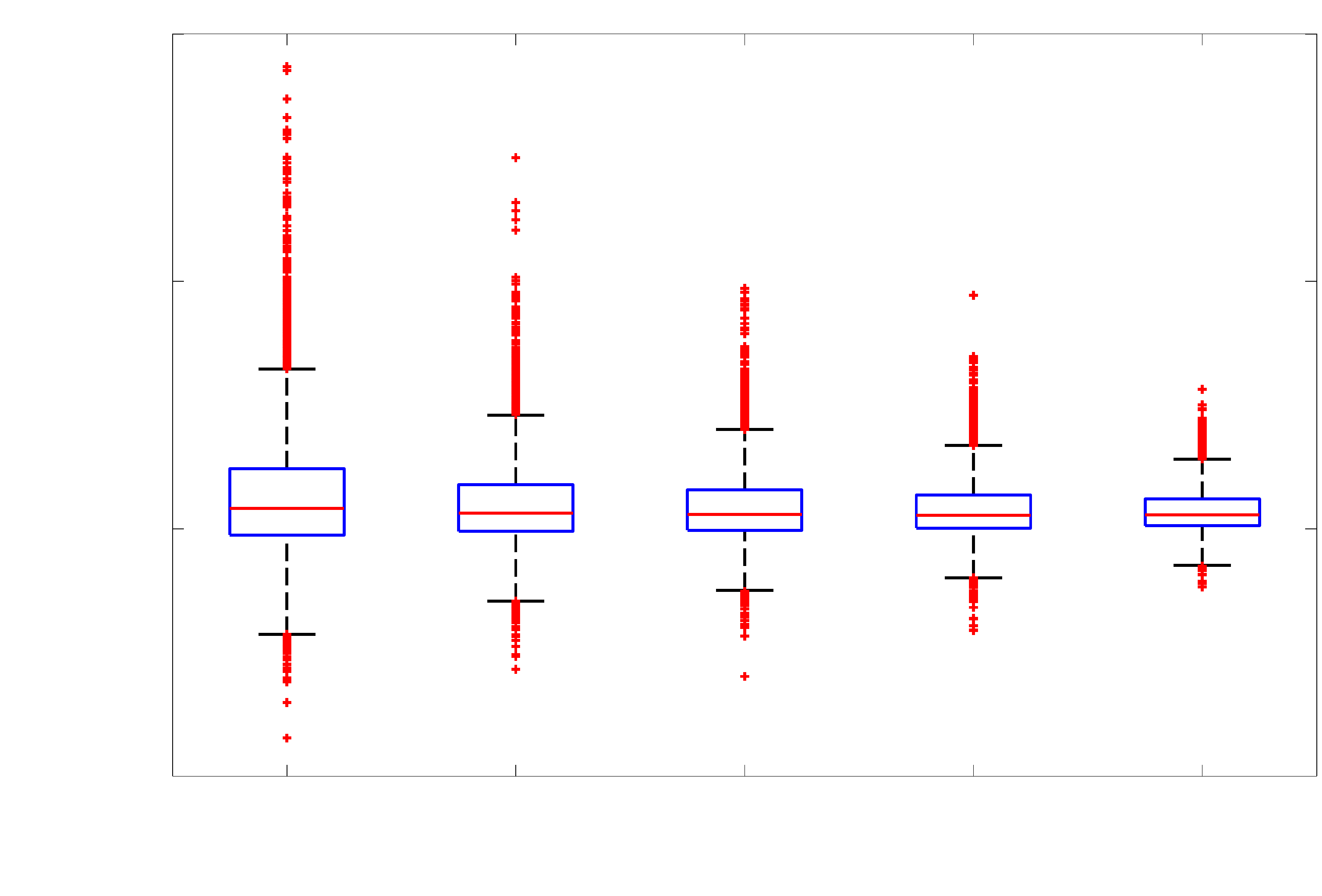}
}
\hfill
\subfloat[Mean and Frequency, $\tau_{01}$]{
\includegraphics[width=0.47\textwidth,clip=true,trim=0 2cm 0 0]{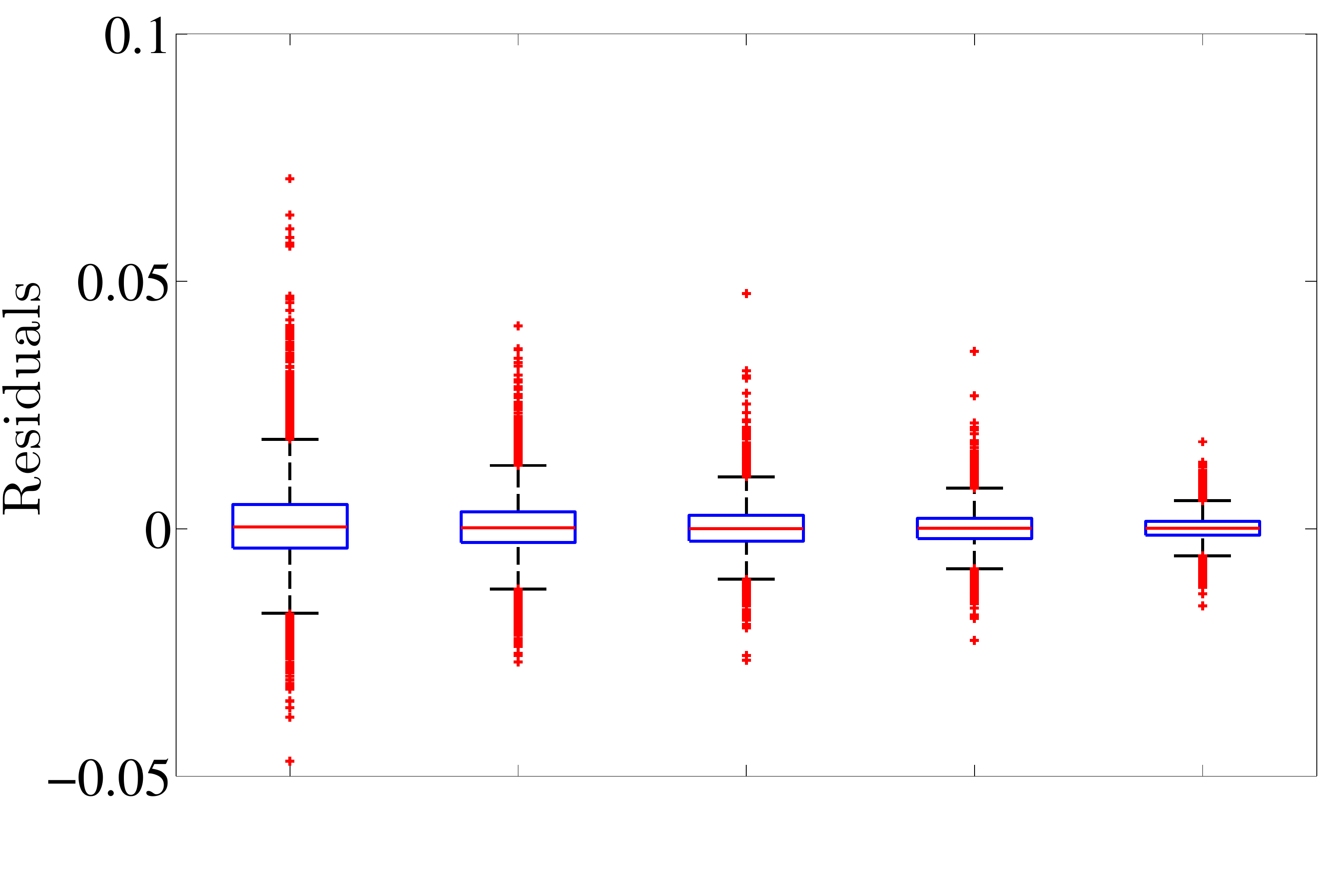}
}
\hfill
\subfloat[Mean and Frequency, $\tau_{10}$]{
\includegraphics[width=0.47\textwidth,clip=true,trim=0 2cm 0 0]{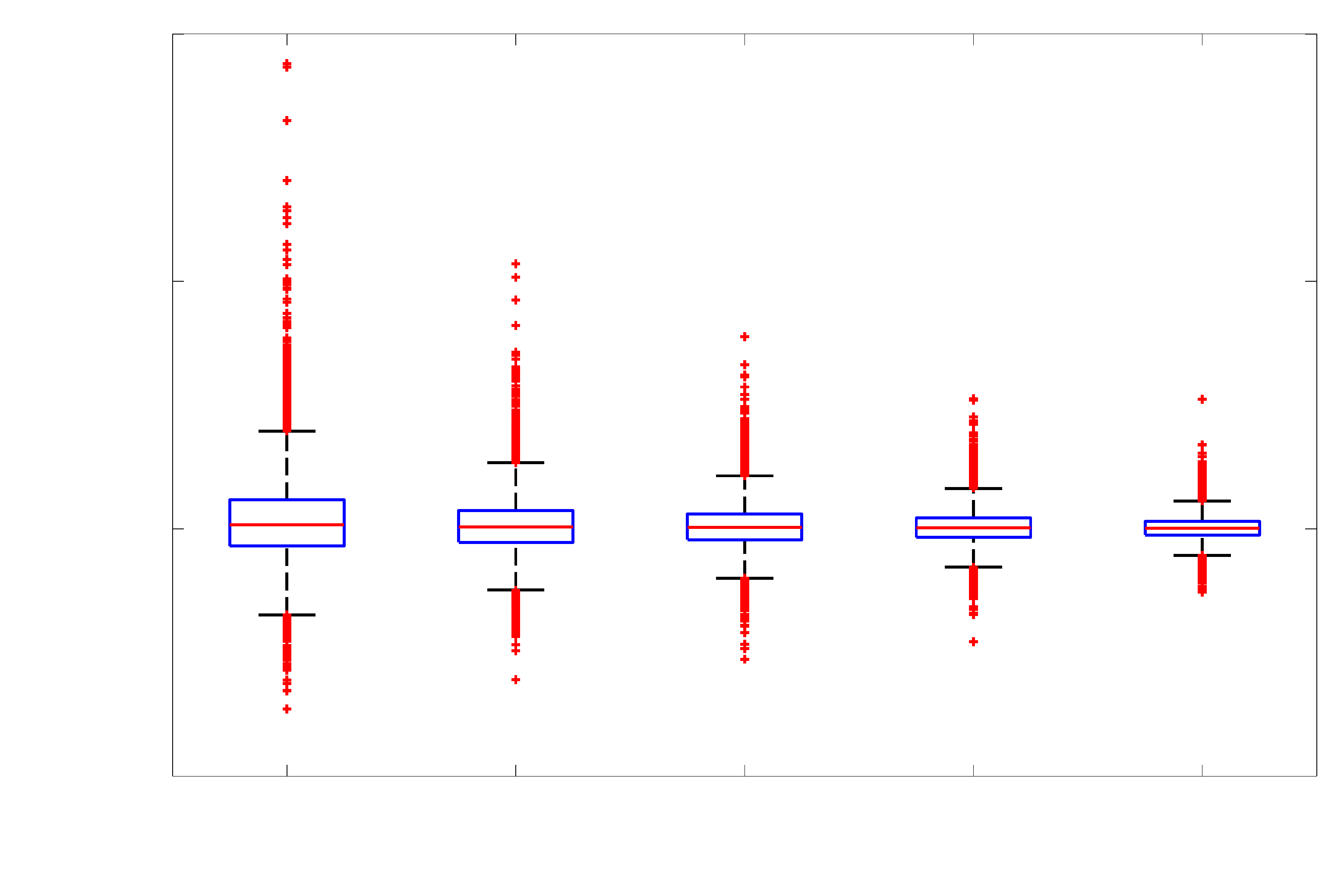}
}
\hfill
\subfloat[\vspace{-\baselineskip} Mean and Cumulative Power, $\tau_{01}$]{
\includegraphics[width=0.47\textwidth]{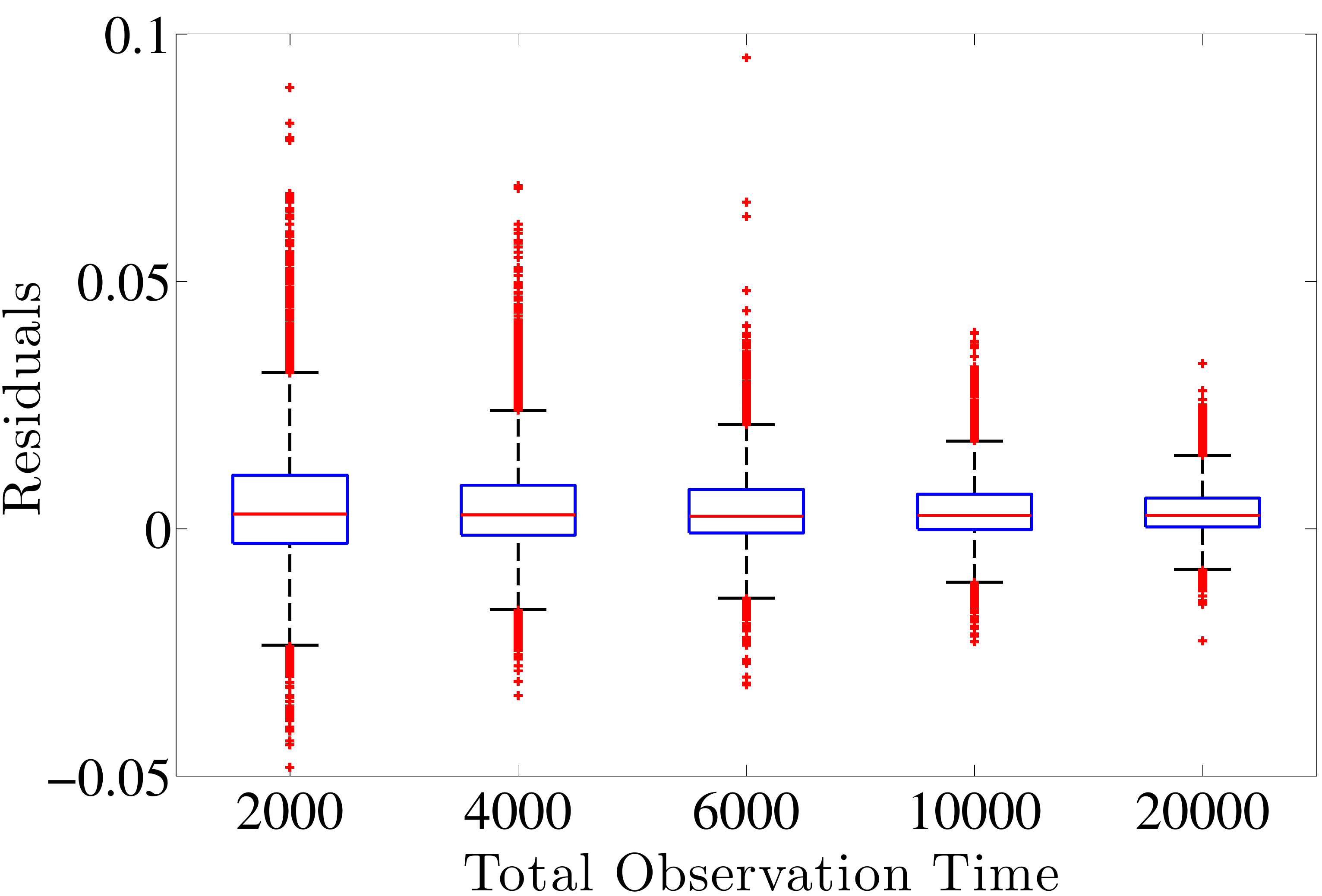}
}
\hfill
\subfloat[Mean and Cumulative Power, $\tau_{10}$]{
\includegraphics[width=0.47\textwidth]{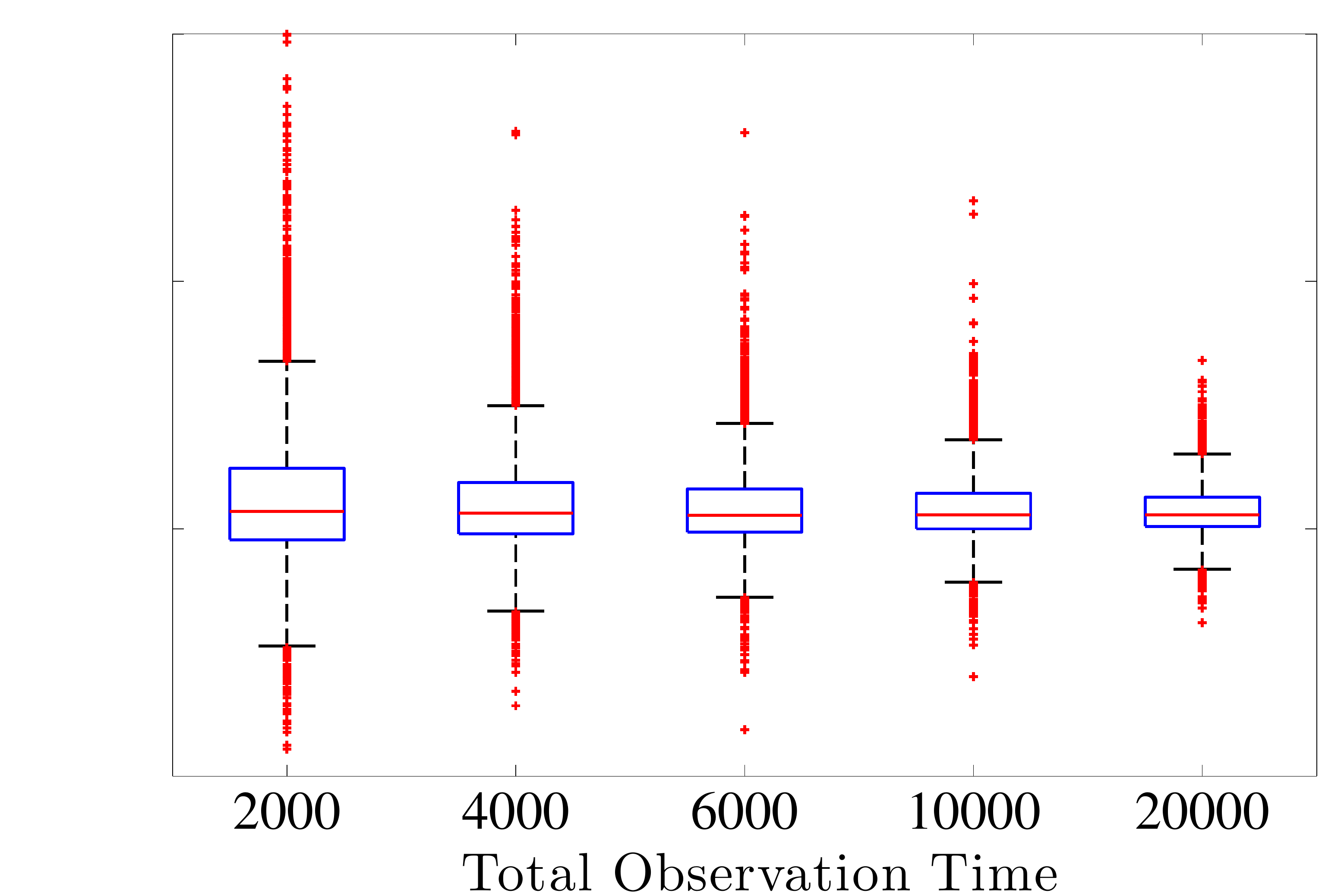}
}
\caption{Boxplots comparing the performance of various couplets. Red markers denote outliers, defined as those values separated from the median by more than 1.5 times the interquartile range.}
\label{fig: boxplots}
\end{figure}

Next we compare the different couplets to one another. \autoref{fig: boxplots} shows boxplots of the residuals for the estimators \eqref{eq: MF}, \eqref{eq: MV}, and \eqref{eq: MS}. These boxplots were generated in the same manner as the results of \autoref{fig: surfaces}, except simulations were run with a varying number of total time steps. As the number of observed time steps increases, estimates become more accurate. This is due to the law of large numbers: As the donkey takes more steps, his observable statistics---mean, variance, frequency, and cumulative power, in our case---approach their theoretical values, improving the accuracy of the inverted couplets.

Based on these boxplots, the mean and frequency couplet produces, on average, the best estimates of $\tau_{01}$ and $\tau_{10}$.  The median residual (shown as the middle, red line) appears to converge to zero quickly, and the interquartile range (the vertical distance between the two blue lines) is narrow even for a small number of time steps. These qualities indicate that the mean and frequency couplet is a more efficient and less biased estimator of $\tau_{01}$ and $\tau_{10}$ than the other two couplets.

%%%%%%%%%%%%%%%%%%%%%%%%%%%%%%%%%%%%%%%%%%%%%%%%%%%

\subsection{Likelihood-Based Estimation} \label{sec: likelihood}
Another method of parameter estimation takes a likelihood-based approach. This approach involves calculating a function, $F(\widehat{\tau_{01}},\widehat{\tau_{10}})$, that measures the likelihood of the observed data matching the parameters $\widehat{\tau_{01}}$ and $\widehat{\tau_{10}}$.  One such function is the log-likelihood function. Given a set of data, the log-likelihood function evaluates the beta probability density function 
\begin{align}
P\left(x | \widehat{\tau_{01}},\widehat{\tau_{10}}\right) = Cx^{\frac{\widehat{\tau_{01}}}{v}-1}\left(1-x\right)^{\frac{\widehat{\tau_{10}}}{v}-1}
\end{align}
at the observed data (the donkey's observed positions). The function then takes the log of each result and sums these new results. Ideally, when this process is performed with the true $\tau_{01}$ and $\tau_{10},$ the majority of data is concentrated in areas of high density, so that $F(\tau_{01},\tau_{10})$ is maximized at $(\tau_{01},\tau_{10}).$ To make this maximum a minimum, we examine $-F(\tau_{01},\tau_{10})$ instead of $F(\tau_{01},\tau_{10}).$ The estimates for $\tau_{01}$ and $\tau_{10}$ are taken to be the location of this minimum.

For example, \autoref{fig: like} graphs the negative log-likelihood function given a simulation of donkey on a line with 10,000 time steps. The input parameters were $\tau_{01} = 0.05$ and $\tau_{10} = 0.08,$ with $v = 0.1.$ The minimum of the negative log-likelihood function, shown in \autoref{fig: like} as a blue dot, is located at $\widehat{\tau_{01}} = 0.045$ and $\widehat{\tau_{10}} = 0.085.$ Thus, estimation errors are quite low. And as with couplet inversion, this likelihood-based technique benefits from greater observation time.

\begin{figure}[h]
\centering
\includegraphics[width=0.6\textwidth]{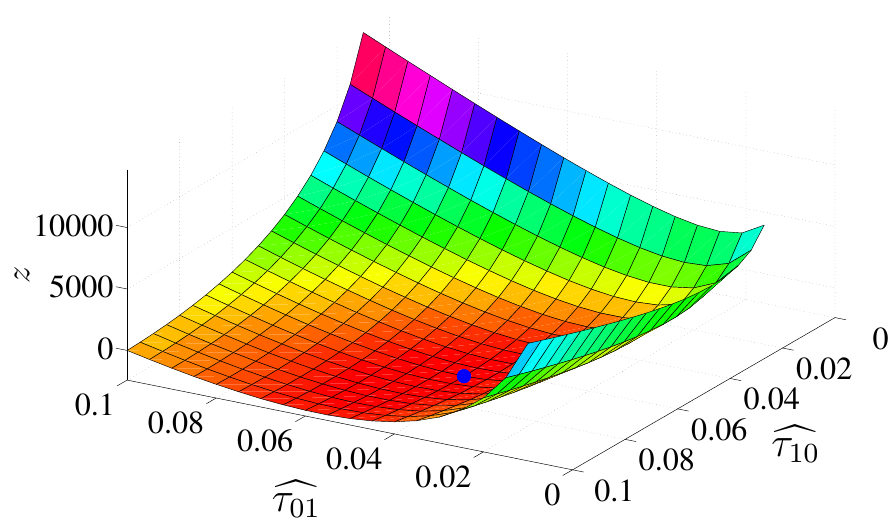} 
\caption{Negative log-likelihood function with $z = -F(\widehat{\tau_{01}},\widehat{\tau_{10}})$}
\label{fig: like}
\end{figure}

%\placetextbox{0.475}{0.434}{\colorbox{black!0}{\textcolor{white}{coverup}}}
%\placetextbox{0.425}{0.44}{$\widehat{\tau_{01}}$}
%\placetextbox{0.65}{0.45}{\colorbox{black!0}{\textcolor{white}{coverup}}}
%\placetextbox{0.66}{0.46}{$\widehat{\tau_{10}}$}

%%%%%%%%%%%%%%%%%%%%%%%%%%%%%%%%%%%%%%%%%%%%%%%%%%%

\subsection{State Detection} \label{sec: state}
While method of moments type estimators and likelihood-based estimators provide reasonable estimates for $\tau_{01}$ and $\tau_{10}$, we are, in fact, able to calculate these parameters directly from our data. By comparing two consecutive positions of the donkey, we have a vector that points in the direction the donkey is moving. If the donkey's position is increasing, we know he is in state 1. By contrast, if the donkey's position is decreasing, we know he is in state 0. By dividing the number of time steps that the donkey changes from state 0 to state 1 by the number of time steps the donkey spends in state 0, we have our result for $\tau_{01}$. A similar computation yields $\tau_{10}$. This state detection technique provides the actual, observed values of $\tau_{01}$ and $\tau_{10}$. Because of the stochastic nature of the system, we do not expect these measured values to exactly match our input values. Nevertheless, by the law of large numbers, a greater number of observed time steps should lend measured values closer to the input values. \autoref{tab: StateDetector1} affirms this expectation.

\begin{table}[h]
\caption{Comparison of input and observed values for $\tau_{ij}$}
\centering
\begin{tabular}{|c|c|c|c|}
\cline{3-4}
\multicolumn{2}{c|}{} & \multicolumn{2}{|c|}{Observed ($\times 10^{-3}$)} \\
	\cline{2-4}
	\multicolumn{1}{c|}{} & Input ($\times 10^{-3}$) & $N = 10000$ & $N = 100000$ \\ \cline{1-4}
	$\tau_{01}$ & 5.00 & 4.73 & 4.90 \\ \hline
	$\tau_{10}$ & 8.00 & 6.83 & 8.09 \\ \hline
\end{tabular}
\label{tab: StateDetector1}
\end{table}

%%%%%%%%%%%%%%%%%%%%%%%%%%%%%%%%%%%%%%%%%%%%%%%%%%%

\section{Donkey in a Triangle}
We now consider a system in which the donkey has more than two states by extending the classic, one-dimensional problem to two dimensions. We let the donkey roam in a triangular pen with vertices located at the coordinates $\left(0,0\right)$, $\left(1,0\right)$, and $\left(0,1\right)$. Let state 0 be the state in which the donkey is moving toward $\left(0,0\right)$, state 1 be the state in which the donkey is moving toward $\left(1,0\right)$, and state 2 be the state in which the donkey is moving toward $\left(0,1\right)$.\footnote{For simplicity, the donkey is always initially placed at $\left(\frac{1}{3},\frac{1}{3}\right)$ and is taken to be in state 0.} The donkey's continuous motion is modeled by the following equations:
			\begin{alignat}{2}
			\text{State 0:\quad} \frac{dx}{dt} &= v(-x) &\quad \frac{dy}{dt} &= v(-y) \\
			\text{State 1:\quad} \frac{dx}{dt} &= v(1-x) & \frac{dy}{dt} &= v(-y) \\
			\text{State 2:\quad} \frac{dx}{dt} &= v(-x) & \frac{dy}{dt} &= v(1-y).
			\end{alignat}
Keeping the notation of the one-dimensional case, the donkey switches from state $i$ to state $j$ at a given time step with probability $\tau_{ij}$.  We now have a Markov process with six parameters to estimate instead of only two.
	
%%%%%%%%%%%%%%%%%%%%%%%%%%%%%%%%%%%%%%%%%%%%%%%%%%%

\subsection{The Markov Process} \label{Donkey in a Triangle: The Markov Process}
Recall that in the one-dimensional case we examined a Markov chain to determine the expected proportion of time the donkey spends in each state. In the two-dimensional case we analyze an analogous process:
	\begin{align}
		p_{n}&=Ap_{n-1},
	\end{align}
	where
	\begin{align}
	A=\begin{bmatrix}
	1-\tau_{01}-\tau_{02}&\tau_{10}&\tau_{20} \\
	\tau_{01}&1-\tau_{10}-\tau_{12}&\tau_{21} \\
	\tau_{02}&\tau_{12}&1-\tau_{20}-\tau_{21}
	\end{bmatrix}
	\end{align}
	is still a column stochastic matrix. As before, $A$ has eigenvalue 1 and associated eigenvector (see \hyperref[app: eigenvector]{Appendix \ref*{app: eigenvector}})
	\begin{align}
	\mathbf{v}=
	\begin{bmatrix}
	\frac{\tau_{21}\tau_{10} + \tau_{12}\tau_{20} + \tau_{10}\tau_{20}}{\tau_{21}\tau_{10} + \tau_{12}\tau_{20} + \tau_{10}\tau_{20} + \tau_{20}\tau_{01} + \tau_{02}\tau_{21} + \tau_{01}\tau_{21} + \tau_{10}\tau_{02} + \tau_{01}\tau_{12} + \tau_{02}\tau_{12}} \\[0.3cm]
	\frac{\tau_{20}\tau_{01} + \tau_{02}\tau_{21} + \tau_{01}\tau_{21}}{\tau_{21}\tau_{10} + \tau_{12}\tau_{20} + \tau_{10}\tau_{20} + \tau_{20}\tau_{01} + \tau_{02}\tau_{21} + \tau_{01}\tau_{21} + \tau_{10}\tau_{02} + \tau_{01}\tau_{12} + \tau_{02}\tau_{12}} \\[0.3cm]
	\frac{\tau_{10}\tau_{02} + \tau_{01}\tau_{12} + \tau_{02}\tau_{12}}{\tau_{21}\tau_{10} + \tau_{12}\tau_{20} + \tau_{10}\tau_{20} + \tau_{20}\tau_{01} + \tau_{02}\tau_{21} + \tau_{01}\tau_{21} + \tau_{10}\tau_{02} + \tau_{01}\tau_{12} + \tau_{02}\tau_{12}} \\
	\end{bmatrix}.
	\end{align}
	In the one-dimensional case, we were able to use the eigenvector associated with the Markov matrix to determine the frequency of the donkey switching between states. By adding just one more state to our system, it becomes difficult to define what is meant by the term ``frequency." We now have several types of transitions, so using frequency as an intuitive, feature statistic is less appealing. Furthermore, as we continue to expand the system, finding a closed form for the eigenvector becomes computationally unrealistic; with just six states, the number of terms in the numerator of an eigenvector coordinate exceeds 1000 (see \hyperref[app: joseph]{Appendix \ref*{app: joseph}}).
		
%%%%%%%%%%%%%%%%%%%%%%%%%%%%%%%%%%%%%%%%%%%%%%%%%%%

	\subsection{Continuous Dynamics}
		Hoping to derive statistics such as mean and variance in the two-dimensional case, we are interested in finding the probability distribution of our system. We let $P(x,y,t)$ be the probability density of the donkey's position at time $t.$ Again, we consider a decomposition of $P$ into the state-dependent, conditional probabilities $P_0(x,y,t)$, $P_1(x,y,t)$, and $P_2(x,y,t)$:
		\begin{align}
		P(x,y,t) = P_0(x,y,t) + P_1(x,y,t) + P_2(x,y,t).
		\end{align} 
		We have the following conservation conditions, which are direct extensions of \eqref{eq: cc1} and \eqref{eq: cc2}:
		\begin{align}
			\frac{\partial P_0}{\partial t} &= -\frac{\partial}{\partial x}\left[v(-x)P_0\right] -  \frac{\partial}{\partial y}\left[v(-y)P_0\right] - \tau_{01}P_0 - \tau_{02}P_0 + \tau_{10}P_1 + \tau_{20}P_2 \label{eq: 2cc1} \\
			\frac{\partial P_1}{\partial t} &= -\frac{\partial}{\partial x}\left[v(1-x)P_1\right] - \frac{\partial}{\partial y}\left[v(-y)P_1\right] - \tau_{10}P_1 - \tau_{12}P_1 + \tau_{01}P_0 + \tau_{21}P_2 \label{eq: 2cc2} \\
			\frac{\partial P_2}{\partial t} &= -\frac{\partial}{\partial x}\left[v(-x)P_2\right] -  \frac{\partial}{\partial y}\left[v(1-y)P_2\right] - \tau_{20}P_2 - \tau_{21}P_2 + \tau_{02}P_0 + \tau_{12}P_1. \label{eq: 2cc3}
		\end{align}
		As in the one-dimensional case, we look for a steady-state solution, when
		\begin{align} 
		\frac{\partial P_0}{\partial t} = \frac{\partial P_1}{\partial t} = \frac{\partial P_2}{\partial t} = 0.
		\end{align}
		That is, the probability distribution of the location of the donkey does not change with time, and we may drop the dependence of $P_0,$ $P_1,$ and $P_2$ on $t.$ Adding \eqref{eq: 2cc1}, \eqref{eq: 2cc2}, and \eqref{eq: 2cc3} in this case, we have
		\begin{align}
			\begin{split}
			0 &= -\frac{\partial}{\partial x}\left[v(-x)P_0\right] - \frac{\partial}{\partial x}\left[v(1-x)P_1\right] - 	\frac{\partial}{\partial x}\left[v(-x)P_2\right] \\
			&\phantom{=} -\frac{\partial}{\partial y}\left[v(-y)P_0\right] - \frac{\partial}{\partial y}\left[v(-y)P_1\right] - \frac{\partial}{\partial y}\left[v(1-y)P_2\right],
			\end{split} \label{eq: conservation}
		\end{align}
		where we have collected the $x$ and $y$ components.
		Assuming total fluxes are balanced in each direction, \eqref{eq: conservation} implies
		\begin{align}
			0 &= v(x)P_0 - v(1-x)P_1 + v(x)P_2 \label{eq: balance1} \\
			0 &= v(y)P_0 + v(y)P_1 - v(1-y)P_2. \label{eq: balance2}
		\end{align}
		Solving for $P_0$ in \eqref{eq: balance1} yields
		\begin{align}
			P_0 = -P_1 + \frac{P_1}{x} - P_2. \label{eq: solved1}
		\end{align}
		Substituting \eqref{eq: solved1} into \eqref{eq: balance2} gives us
		\begin{align}
			0 &= v(y)\left(-P_1 + \frac{P_1}{x} - P_2\right) + v(y)P_1 - v(1-y)P_2 \\
			P_1 &= \frac{xP_2}{y}. \label{eq: solved2}
		\end{align}
		Using \eqref{eq: solved2} in \eqref{eq: solved1}, we see
		\begin{align}
			P_0 &= -\frac{xP_2}{y} + \frac{P_2}{y} - P_2. \label{eq: solved3}
		\end{align}
		We then apply \eqref{eq: solved2} and \eqref{eq: solved3} to \eqref{eq: 2cc3}, the left-hand side of which we have assumed to be 0:
		\begin{align}
			\begin{split}
			0 &= -\frac{\partial}{\partial x}\left[v(-x)P_2\right] - \frac{\partial}{\partial y}\left[v(1-y)P_2\right] - 	\tau_{20}P_2 - \tau_{21}P_2 \label{eq11} \\
			&\phantom{=} + \tau_{02}\left(-\frac{xP_2}{y} + \frac{P_2}{y} - P_2\right) + \tau_{12}\frac{xP_2}{y}. \end{split}
		\end{align}
		Applying the derivative operators using the product rule and rearranging terms, we arrive at the following differential equation:
		\begin{align}
			-x\frac{\partial P_2}{\partial x} + (1-y)\frac{\partial P_2}{\partial y} = \frac{P_2}{v}\left(-\tau_{20} - \tau_{21} + \tau_{02} + \frac{\tau_{02}}{y} + \frac{(\tau_{12} - \tau_{02})x}{y} + 2v\right).\label{eq: pde}
		\end{align}
		
Although we identify \eqref{eq: pde} as a first-order, semilinear partial differential equation \linebreak \citep{pde}, we are currently not able to solve it for $P_2$. In fact, we are not sure this is a well-posed problem. Unfortunately, this means we cannot apply the same methods as before to find expressions for moments such as mean and variance. Moreover, we have no probability distribution with which to perform the same likelihood-based scheme as described in \ref{sec: likelihood}.
		
%%%%%%%%%%%%%%%%%%%%%%%%%%%%%%%%%%%%%%%%%%%%%%%%%%%
				
\subsection{State Detection} \label{sec: detection}
As some of the methods applied to the donkey on a line system become intractable in a multi-dimensional extension, we turn to state detection as in \ref{sec: state}.  Before applying a method of state detection to the donkey in a triangular pen, however, we check that his movements are always linear. Consider the case when the donkey is in state 0. As stated earlier, his movement in the $x$ direction is given by $\frac{dx}{dt}=v\left(-x\right)$, and his movement in the $y$ direction is given by $\frac{dy}{dt}=v\left(-y\right)$. Dividing these, we find 
		\begin{align}
		\frac{\frac{dy}{dt}}{\frac{dx}{dt}}=\frac{dy}{dx}=\frac{v\left(-y\right)}{v\left(-x\right)}=\frac{y}{x}. \label{eq: slope}
		\end{align}
		Since the slope given by \eqref{eq: slope} is exactly in the direction of $(0,0),$ we see that the donkey's movement is linear. A similar computation can be performed for states 1 and 2 to show his movement is linear in every state. From this fact, we build a state detector for the donkey using the following algorithm:
		\begin{enumerate}
			\item Compare two consecutive positions of the donkey to make a vector pointing in the direction of the donkey's current movement; call it $\mathbf{a}$.
			\item Construct vectors pointing from the donkey's current position to each of the coordinates of the triangle; call them $\mathbf{b}_i$ for $i = 0,1,2$.
			\item Compare $\mathbf{a}$ to each $\mathbf{b}_i$ by calculating the cosine of the angle between them via the formula 
			\begin{align}
			\cos \theta = \frac{\mathbf{a}\cdot\mathbf{b}_i}{\|\mathbf{a}\|\|\mathbf{b}_i\|}. \label{eq: cos}
			\end{align} The $\mathbf{b}_i$ that yields $\cos \theta =1$ is the vector that is parallel to the donkey's movement. Thus, the donkey is in state $i$.
			\item Repeat this process for all observed data.
		\end{enumerate}
		
Now that we have a state detector, we can use it to find each $\tau_{ij}$ empirically. For example, if we want to estimate $\tau_{10}$, we first use the state detector to find the number of time steps that the donkey spends in state 1. Next, find each time step that the donkey changes from state 1 to state 0. By dividing these results, we have our estimate for $\tau_{10}$. This process is easily seen to generalize to any convex geometry in which the donkey can roam.

Because of the stochastic nature of our system, we do not expect our observed $\tau_{ij}$'s to match perfectly with our input $\tau_{ij}$'s. However, by the law of large numbers, we expect that they will converge as the number of time steps increases. \autoref{tab: StateDetector} shows the results from a simulation with $v=0.01$ and varied total observed time steps $N$. As expected, a greater number of observed time steps yields, on average, observed parameter values closer to input parameter values.
		\begin{table}[h]
		\caption{Comparison of input and observed values for $\tau_{ij}$}
		\centering
		\begin{tabular}{|c|c|c|c|}
		\cline{3-4}
		\multicolumn{2}{c|}{} & \multicolumn{2}{|c|}{Observed ($\times 10^{-3}$)} \\
			\cline{2-4}
			\multicolumn{1}{c|}{} & Input ($\times 10^{-3}$) & $N = 10000$ & $N = 100000$ \\ \cline{1-4}
			$\tau_{01}$ & 1.00 & 0.76 & 0.81 \\ \hline
			$\tau_{02}$ & 6.00 & 5.83 & 6.56\\ \hline
			$\tau_{10}$ & 2.00 & 3.52 & 1.71 \\ \hline
			$\tau_{12}$ & 3.00 & 2.24 & 3.15 \\ \hline
			$\tau_{20}$ & 4.00 & 5.13 & 3.95 \\ \hline
			$\tau_{21}$ & 5.00 & 5.13 & 5.02 \\ \hline
		\end{tabular}
		\label{tab: StateDetector}
		\end{table}

%%%%%%%%%%%%%%%%%%%%%%%%%%%%%%%%%%%%%%%%%%%%%%%%%%%

\section{Noisy Data}
The state detector solves our two-dimensional extension of Buridan's ass, as it allows us to directly measure the parameters of the system. Although these measured values may not be equal to the input parameters, they are the values we wish to obtain as we continue our exploration of donkey in a triangle. They are the observed, as opposed to theoretical, probabilities of transitioning, making them the target estimates for further simulations. 

A system in which we have perfect measurements of the donkey's trajectory is not realistic. Real-world observations of dynamical systems involve some level of error, or noise. By adding noise to our system, we obtain data that is comparable to real-world measurements of dynamical systems. We now shift our focus to managing noise in our data, with the goal of obtaining reasonable estimates for the $\tau_{ij}$'s.

To compare fairly the performance of the techniques discussed, we apply each to the same data set. This data was taken from the simulation used to construct \autoref{tab: StateDetector}. From here on we will assume we observe the donkey for $N=10000$ time steps, so that the target estimates for the various $\tau_{ij}$ are those shown in the middle column of \autoref{tab: StateDetector}. Independent and identically distributed noise was added to each coordinate of the donkey's trajectory, at each time step. This noise was taken from a normal distribution with a mean of 0 and a standard derivation of $0.01.$

%%%%%%%%%%%%%%%%%%%%%%%%%%%%%%%%%%%%%%%%%%%%%%%%%%%

\subsection{Linear Regression}
Linear regression is a statistical technique that aids in the study of linear relationships between variables. By modeling a set of data by a line, linear regression has potential to remove artifacts of noise. Outlined below is our application of linear regression to state detection:
\begin{enumerate}
\item Consider the donkey's observed positions in the $x$ and $y$ directions separately. We let $\mathbf{x}$ and $\mathbf{y}$ be vectors containing the positions of the donkey in these directions, so that $\left(x_j,y_j\right)$ is the donkey's observed location at the $j$th time step. We let $\mathbf{t}$ be the vector containing the times at which observations are made. In general, we assume $t_j = j-1.$
\item For each $j = 1,\dots,N+1,$ where $N$ is the number of observed time steps, we calculate the lines of best fit to the data sets $\left(t_k,x_k\right)_{k=j}^{j+W}$ and $\left(t_k,y_k\right)_{k=j}^{j+W}$, where $W$ is the so-called viewing window. That is, $W$ is the number of time steps from the current position considered in linear regression. (If $j+W$ exceeds the length of $\mathbf{x}$ or $\mathbf{y},$ then $W$ is taken to be the largest possible viewing window.)
\item We let $a_1$ be the slope of the best-fit line in the $x$ direction, and $a_2$ be the slope of the best-fit line in the $y$ direction. We set $\mathbf{a} = \left[a_1, a_2\right]^T$.
\item Next construct vectors pointing from $(\overline{x},\overline{y}),$ where $\overline{x} = \frac{1}{W+1}\sum_{k=j}^{j+W} x_k$, and $\overline{y} = \frac{1}{W+1}\sum_{k=j}^{j+W} y_k,$ to each of the coordinates of the triangle; call these vectors $\mathbf{b}_i$ for $i = 0, 1, 2.$ 
\item Proceed as in \ref{sec: detection}, now accounting for noise by taking the $i$ that yields the greatest value in \eqref{eq: cos}.
\end{enumerate}

Parameter estimates obtained using this algorithm are displayed in \autoref{tab: linreg}. Although increasing the viewing window $W$ leads to more reasonable estimates, in no case does linear regression provide reliable information. Each estimate's absolute relative error is also shown in blue in \autoref{tab: linreg}, where the error is computed based on the observed $\tau_{ij}$ values of \autoref{tab: StateDetector}. Clearly the noise present in the system must be dealt with in a different manner.

\begin{table}[ht] \centering
\caption{Estimated values $(\times 10^{-3})$ for $\tau_{ij}$ and \textcolor{blue}{relative error} after adding noise, using linear regression with different viewing windows}
\begin{tabular}{|c|c|c|c|c|c|c|}
	\cline{2-4}
	\multicolumn{1}{c|}{} & $W = 1$ & $W = 3$ & $W = 5$ \\
	\cline{1-4}
	$\tau_{01}$ & 415.67 \textcolor{blue}{$(54471.04\%)$} 
	& 248.34 \textcolor{blue}{$(32581.83\%)$} & 160.30 \textcolor{blue}{$(20995.21\%)$} \\ \hline
	$\tau_{02}$ & 416.29 \textcolor{blue}{$(7045.70\%)$} 
	& 239.39 \textcolor{blue}{$(4009.18\%)$} & 131.80 \textcolor{blue}{$(2162.38\%)$} \\ \hline
	$\tau_{10}$ & 372.51 \textcolor{blue}{$(10492.96\%)$} 
	& 214.55 \textcolor{blue}{$(6001.10\%)$} & 138.62 \textcolor{blue}{$(3841.92\%)$} \\ \hline
	$\tau_{12}$ & 412.57 \textcolor{blue}{$(18336.12\%)$} 
	& 193.85 \textcolor{blue}{$(8562.11\%)$} & 100.55 \textcolor{blue}{$(4392.25\%)$} \\ \hline
	$\tau_{20}$ & 372.17 \textcolor{blue}{$(7154.76\%)$} 
	& 206.50 \textcolor{blue}{$(3925.43\%)$} & 122.19 \textcolor{blue}{$(2281.89\%)$} \\ \hline
	$\tau_{21}$ & 402.30 \textcolor{blue}{$(7742.08\%)$} 
	& 191.39 \textcolor{blue}{$(3630.75\%)$} & 94.52 \textcolor{blue}{$(1742.60\%)$} \\ \hline 
	\multicolumn{1}{c|}{\vspace{-1ex}} & \multicolumn{3}{c|}{\vspace{-1ex}} \\ \cline{2-4}
	\multicolumn{1}{c|}{} & $W = 10$ & $W = 15$ & $W = 20$ \\
	\cline{1-4}
	$\tau_{01}$ & 54.47 \textcolor{blue}{$(7068.42\%)$} & 28.32 \textcolor{blue}{$(3626.27\%)$}
	& 22.17 \textcolor{blue}{$(2817.15\%)$} \\ \hline
	$\tau_{02}$ & 53.31 \textcolor{blue}{$(815.01\%)$} & 29.16 \textcolor{blue}{$(400.47\%)$}
	& 22.17 \textcolor{blue}{$(280.50\%)$} \\ \hline
	$\tau_{10}$ & 55.13 \textcolor{blue}{$(1467.63\%)$} & 36.78 \textcolor{blue}{$(945.85\%)$}
	& 29.14 \textcolor{blue}{$(728.60\%)$} \\ \hline
	$\tau_{12}$ & 40.34 \textcolor{blue}{$(1702.84\%)$} & 25.68 \textcolor{blue}{$(1047.60\%)$}
	& 17.93 \textcolor{blue}{$(701.28\%)$} \\ \hline
	$\tau_{20}$ & 57.25 \textcolor{blue}{$(1015.92\%)$} & 26.85 \textcolor{blue}{$(423.31\%)$}
	& 20.81 \textcolor{blue}{$(305.74\%)$} \\ \hline
	$\tau_{21}$ & 37.06 \textcolor{blue}{$(622.41\%)$} & 29.29 \textcolor{blue}{$(470.88\%)$}
	& 20.51 \textcolor{blue}{$(299.77\%)$} \\ \hline
\end{tabular}
\label{tab: linreg}
\end{table}

%%%%%%%%%%%%%%%%%%%%%%%%%%%%%%%%%%%%%%%%%%%%%%%%%%%

\subsection{Denoising Techniques}
An alternative approach to linear regression is to denoise the observed data before applying state detection. That is, through a chosen method, we transform our observed, noisy data into a data set which should represent the actual trajectory of the donkey. We now explore various methods of denoising and compare the accuracy of the resulting parameter estimates. Each method is presented as applied to a one-dimensional data set. That is, we denoise the data vectors $\mathbf{x}$ and $\mathbf{y}$ (each one containing $N+1$ coordinates) separately and then apply the state detector presented in \ref{sec: detection} to this denoised data, thereby producing estimates of the various $\tau_{ij}.$

%%%%%%%%%%%%%%%%%%%%%%%%%%%%%%%%%%%%%%%%%%%%%%%%%%%

\subsubsection{Locally Weighted Polynomial Regression}
The first denoising technique we explore is locally weighted polynomial regression (LWPR). LWPR is a method of fitting a curve to a data set using low-degree polynomials \citep{lwpr}. To utilize the versatility of polynomials while avoiding the Runge phenomenon that plagues high-degree interpolation \citep{runge}, LWPR fits polynomials to subsets of the data rather than fitting a single polynomial to the entire set.  A smoothing process is then applied to join the local polynomials into a single curve. This process is governed by a smoothing parameter, $0<h\leq1,$ which is the percentage of the total data that is considered in smoothing about a given data point \citep{smooth}.

In the same simulation as before, LWPR was performed using quadratic polynomials and $h=0.005$. Results are shown in \autoref{tab: lwpr}. Despite being far more reasonable than the estimates provided by linear regression, those provided by LWPR are still poor. The estimates are nearly correct in the ordering of the parameters (i.e.~which are the smallest, and which are the largest), but they are, in every case, overestimates. Such distortion of the parameters is caused by over-detection of state transitions. The performance of LWPR may be improved by utilizing polynomials of a different degree and adjusting the parameter $h.$

\begin{table}[ht] \centering
\caption{Estimated values for $\tau_{ij}$ and \textcolor{blue}{relative error} after adding noise, using locally weighted polynomial regression}
\begin{tabular}{|c|c|c|c|}
	\cline{2-4}
	\multicolumn{1}{c|}{} & Input ($\times 10^{-3}$) & Observed ($\times 10^{-3}$) & Estimate ($\times 10^{-3}$) \\ \cline{1-4}
	$\tau_{01}$ & 1.00 & 0.76 & 3.76 \textcolor{blue}{$(395.40\%)$} \\ \hline
	$\tau_{02}$ & 6.00 & 5.83 & 9.41 \textcolor{blue}{$(61.54\%)$} \\ \hline
	$\tau_{10}$ & 2.00 & 3.52 & 8.48 \textcolor{blue}{$(141.06\%)$} \\ \hline
	$\tau_{12}$ & 3.00 & 2.24 & 4.89 \textcolor{blue}{$(118.55\%)$} \\ \hline
	$\tau_{20}$ & 4.00 & 5.13 & 7.16 \textcolor{blue}{$(39.50\%)$} \\ \hline
	$\tau_{21}$ & 5.00 & 5.13 & 8.40 \textcolor{blue}{$(63.76\%)$} \\ \hline
\end{tabular}
\label{tab: lwpr}
\end{table}

%%%%%%%%%%%%%%%%%%%%%%%%%%%%%%%%%%%%%%%%%%%%%%%%%%%

\subsubsection{Wavelet Filtering}
Wavelets are functions that together form an orthonormal basis of an infinite-dimensional vector space\footnote{Often this vector space is $L^2(\mathbb{R}),$ the space of square-integrable functions on the real line.} consisting of other functions. There are many different wavelet families, each suitable to represent a unique class of functions. While providing the convenience of orthonormal bases, wavelets have the added utility of decomposing functions by both frequency and scale \citep{wavelets}. That is, frequency can be analyzed over subsets of a function's domain, in contrast to decomposition into a Fourier basis, where a function's entire domain is considered in each basis coefficient.

Wavelets can be used to filter noisy signals by decomposing them into their orthonormal basis representation and adjusting coefficients polluted by noise. In our case, we use the symlets 8 wavelet basis (up to level 6) with multi-level, soft thresholding. \autoref{tab: wavelet} shows the parameter estimations resulting from this denoising. The overestimates suggest that wavelet filtering still allows over-detection of state transitions. There are many other wavelet families, however, and a variety of thresholding techniques. Wavelet filtering cannot be discarded entirely before performing a thorough exploration of the many possible combinations.

\begin{table}[ht] \centering
\caption{Estimated values for $\tau_{ij}$ and \textcolor{blue}{relative error} after adding noise, using wavelet filtering}
\begin{tabular}{|c|c|c|c|}
	\cline{2-4}
	\multicolumn{1}{c|}{} & Input ($\times 10^{-3}$) & Observed ($\times 10^{-3}$) & Estimate ($\times 10^{-3}$) \\ \cline{1-4}
	$\tau_{01}$ & 1.00 & 0.76 & 22.52 \textcolor{blue}{$(2863.47\%)$} \\ \hline
	$\tau_{02}$ & 6.00 & 5.83 & 49.21\textcolor{blue}{$(744.66\%)$} \\ \hline
	$\tau_{10}$ & 2.00 & 3.52 & 28.40 \textcolor{blue}{$(707.54\%)$} \\ \hline
	$\tau_{12}$ & 3.00 & 2.24 & 21.06 \textcolor{blue}{$(841.05\%)$} \\ \hline
	$\tau_{20}$ & 4.00 & 5.13 & 51.39 \textcolor{blue}{$(901.80\%)$} \\ \hline
	$\tau_{21}$ & 5.00 & 5.13 & 22.64 \textcolor{blue}{$(341.27\%)$} \\ \hline
\end{tabular}
\label{tab: wavelet}
\end{table}

%%%%%%%%%%%%%%%%%%%%%%%%%%%%%%%%%%%%%%%%%%%%%%%%%%%

\subsubsection{Butterworth Filtering}
The Butterworth filter is a signal processing tool used to denoise a system by permitting and dampening certain specified frequencies present in the data \citep{butterworth}. There are three types of Butterworth filters: lowpass, highpass, and bandpass.  A lowpass filter permits frequencies below a specified cutoff frequency $\omega_c$ but gradually dampens those exceeding $\omega_c.$ In contrast, a highpass filter permits frequencies above the cutoff but dampens those below it.  A bandpass filter has both lowpass and highpass cutoff frequencies. 

The type of filter we apply depends on which frequencies are most prevalent in the observed data. \autoref{fig: frequencies} shows that, at least in our sample simulation, such frequencies are low in the spectrum. The $x$ coordinate of the donkey's trajectory is shown in \hyperref[subfig: frequencies1]{Figure \ref*{subfig: frequencies1}}, and the discrete Fourier transform $F$ of $\mathbf{x}$, defined by
\begin{align}
F(\omega) = \sum_{k = 1}^{N+1} x_k e^{\frac{-2\pi \sqrt{-1} \omega(k-1)}{N+1}},
\end{align}
is shown in \hyperref[subfig: frequencies2]{Figure \ref*{subfig: frequencies2}}.  As the majority of energy is contained in the lower frequencies, a lowpass Butterworth filter is suitable for our problem. Our goal is to eliminate noise that may dominate high frequencies while relying on the high energy in the low frequencies to suppress the effect of noise there.

\begin{figure}[ht]
\subfloat[Noiseless $\mathbf{x}$ against time]{
\includegraphics[width=0.47\textwidth]{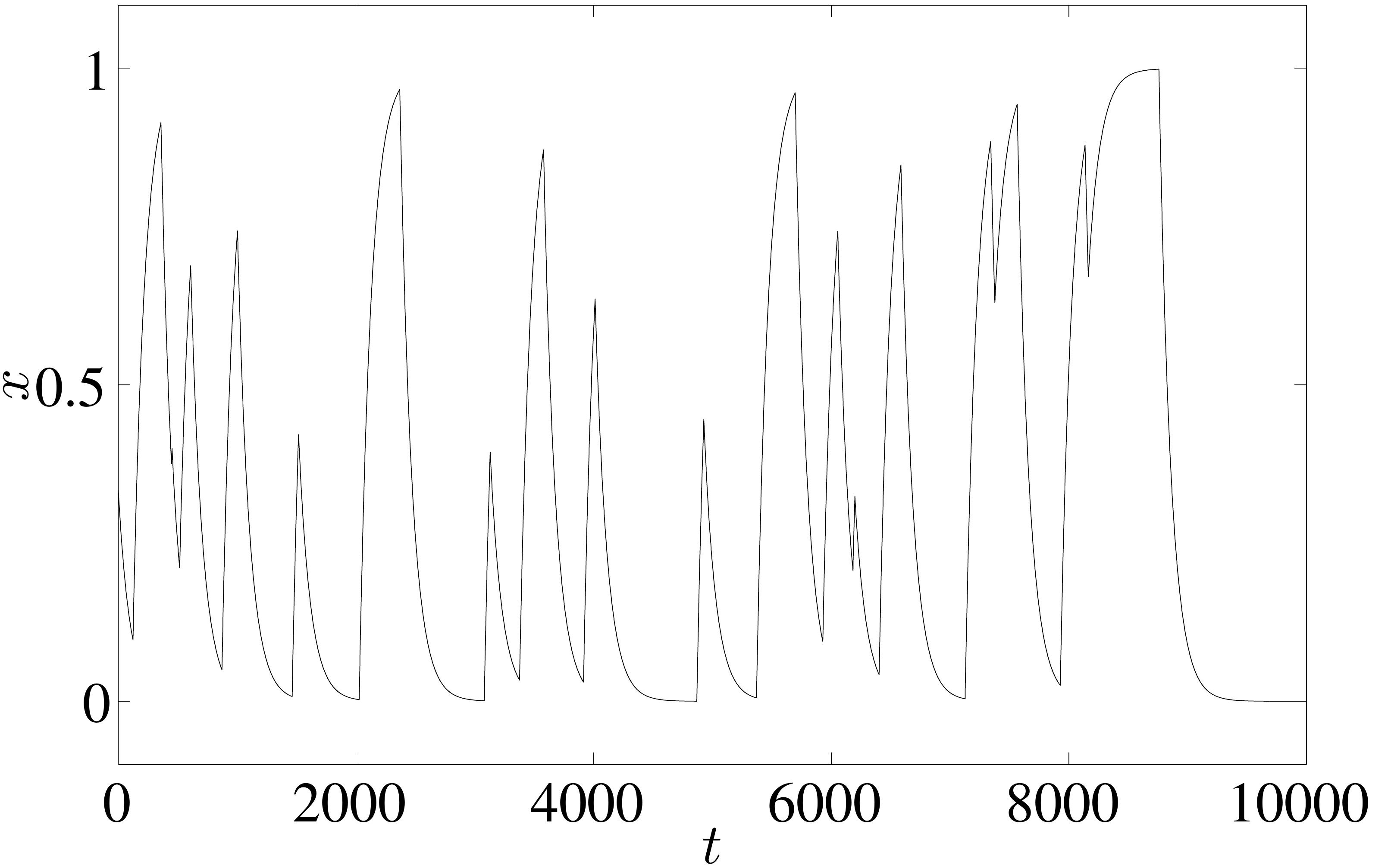}
\label{subfig: frequencies1}	
}
\hfill
\subfloat[Discrete Fourier transform of $\mathbf{x}$]{
\includegraphics[width=0.47\textwidth]{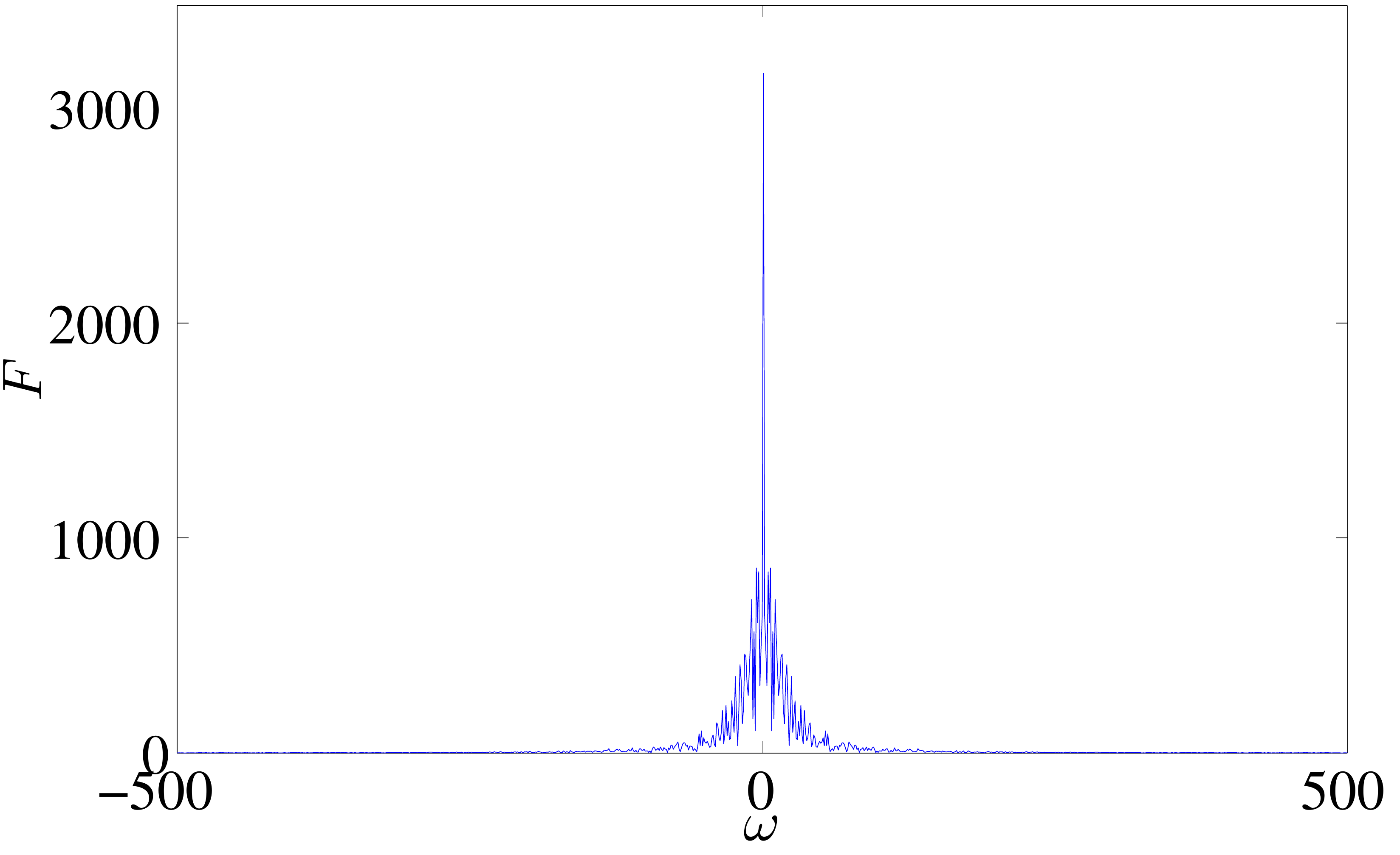}
\label{subfig: frequencies2}
}
\caption{Identification of predominant frequencies in the donkey's trajectory}
\label{fig: frequencies}
\end{figure}

To denoise the donkey's trajectory, a fifth order lowpass Butterworth filter was used, with $\omega_c = 100.$ The resulting parameter estimates are given in \autoref{tab: butterworth}. These results are easily the best seen thus far.

\begin{table}[ht] \centering
\caption{Estimated values for $\tau_{ij}$ and \textcolor{blue}{relative error} after adding noise, using butterworth filtering}
\begin{tabular}{|c|c|c|c|}
	\cline{2-4}
	\multicolumn{1}{c|}{} & Input ($\times 10^{-3}$) & Observed ($\times 10^{-3}$) & Estimate ($\times 10^{-3}$) \\ \cline{1-4}
	$\tau_{01}$ & 1.00 & 0.76 & 2.68 \textcolor{blue}{$(253.19\%)$} \\ \hline
	$\tau_{02}$ & 6.00 & 5.83 & 4.29 \textcolor{blue}{$(26.29\%)$} \\ \hline
	$\tau_{10}$ & 2.00 & 3.52 & 4.14 \textcolor{blue}{$(17.77\%)$} \\ \hline
	$\tau_{12}$ & 3.00 & 2.24 & 2.55 \textcolor{blue}{$(13.89\%)$} \\ \hline
	$\tau_{20}$ & 4.00 & 5.13 & 4.57 \textcolor{blue}{$(12.95\%)$} \\ \hline
	$\tau_{21}$ & 5.00 & 5.13 & 3.51 \textcolor{blue}{$(31.60\%)$} \\ \hline
\end{tabular}
\label{tab: butterworth}
\end{table}

%%%%%%%%%%%%%%%%%%%%%%%%%%%%%%%%%%%%%%%%%%%%%%%%%%%

\subsubsection{Total Variation Denoising}
The last denoising approach we considered is a technique commonly used in image processing: total variation (TV) \citep{TV}. Unlike the Butterworth filter, TV is not spectral-based; that is, it does not involve decomposition into a Fourier basis. Rather, to denoise a vector $\mathbf{x},$ TV seeks to solve the optimization problem
\begin{align}
\argmin_{\hat{\mathbf{x}}} \frac{\gamma}{2} \| \hat{\mathbf{x}} - \mathbf{x} \|_2^2 + \| \nabla \hat{\mathbf{x}} \|_1,
\label{eq4.5.1}
\end{align}
where
\begin{align}
\nabla = \begin{bmatrix}
0&0&0&\cdots&0&0 \\
-1&1&0&\cdots&0&0 \\
0&-1&1&\cdots&0&0 \\
\vdots&\vdots&\vdots&\ddots&\vdots&\vdots \\
0&0&0&\cdots&-1&1
\end{bmatrix}
\end{align}
is a discrete derivative operator.\footnote{$\nabla$ is shown with entries of $\pm 1$ since we assume measurements are taken a unit of time apart. The matrix, however, can be generalized for nonunit and nonuniform spacing, each row with entries of $\pm \Delta t$ for a unique $\Delta t.$} The goal is to capture the overall behavior of $\mathbf{x}$ with an estimate $\hat{\mathbf{x}},$ while removing small, spurious oscillations that are merely artifacts of noise. Hence, we minimize the sum of a fitting term and a penalty term that discretely approximates the integral of the absolute value of the derivative (this integral is called a function's {\it total variation}). Our choice of $\gamma$ should reflect the relative importance of minimizing each term. With greater noise, oscillations become more pronounced, so $\gamma$ would be placed at a lower value to assign more weight to minimizing total variation.

Difficulty lies is solving \eqref{eq4.5.1}, due to the non-differentiability of the $\ell1$-norm.\footnote{The $\ell1$-norm is defined as $\| \mathbf{x} \|_1 = \sum_k \lvert x_k \rvert.$ The $\ell2$-norm is defined as $\| \mathbf{x} \|_2 = \sqrt{\sum_k \lvert x_k \rvert^2}.$} To resolve this issue, we refer to \cite{bregman}, in which the authors rephrase the unconstrained minimization into a constrained one, in the process splitting the $\ell1$ and $\ell2$ portions of the problem. In this new formulation, the authors solve the minimization using so-called Split Bregman iteration, which requires the convexity of the $\ell1$- and $\ell2$-norms. We demonstrate below this process as applied to \eqref{eq4.5.1}.

We first reformulate \eqref{eq4.5.1} as an equivalent unconstrained problem:
\begin{align}
\argmin_{\hat{\mathbf{x}}} \frac{\gamma}{2} \| \hat{\mathbf{x}} - \mathbf{x} \|_2^2 + \| \mathbf{d} \|_1 \text{ such that } \mathbf{d} = \nabla \hat{\mathbf{x}}. \label{eq4.5.3}
\end{align}
The constraint in \eqref{eq4.5.3}, though, is weakly enforced, introducing a second fitting term with weight parameter $\lambda$:
\begin{align}
\argmin_{\hat{\mathbf{x}}} \frac{\gamma}{2} \| \hat{\mathbf{x}} - \mathbf{x} \|_2^2 + \frac{\lambda}{2} \| \mathbf{d} - \nabla \hat{\mathbf{x}}\|_2^2 + \| \mathbf{d} \|_1. \label{eq4.5.4}
\end{align}
The optimization as presented in \eqref{eq4.5.4} is now amenable to Split Bregman iteration, which initializes $\mathbf{d}_0 = \mathbf{b}_0 = \mathbf{0}$ and follows a two step process for $i = 1,\dots,n$:
\begin{subequations}  \label{eq4.5.5}
\begin{align}
&\text{Step 1: } (\hat{\mathbf{x}}_{i+1},\mathbf{d}_{i+1}) = \argmin_{\hat{\mathbf{x}},\mathbf{d}} \frac{\gamma}{2} \| \hat{\mathbf{x}} - \mathbf{x} \|_2^2 + \frac{\lambda}{2}\| \mathbf{d} - \nabla \hat{\mathbf{x}} - \mathbf{b}_i \|_2^2 + \| \mathbf{d} \|_1 \\
&\text{Step 2: } \mathbf{b}_{i+1} = \mathbf{b}_i + \nabla \hat{\mathbf{x}}_{i+1} - \mathbf{d}_{i+1}.
\end{align}
\end{subequations}
For step 1, the algorithm solves for $\hat{\mathbf{x}}$ and $\mathbf{d}$ separately. First,
\begin{align}
\hat{\mathbf{x}}_{i+1} = \argmin_{\hat{\mathbf{x}}} \frac{\gamma}{2} \| \hat{\mathbf{x}} - \mathbf{x} \|_2^2 + \frac{\lambda}{2}\| \mathbf{d} - \nabla \hat{\mathbf{x}} - \mathbf{b}_i \|_2^2, \label{eq4.5.6}
\end{align}
which can be solved by differentiating with respect to $\hat{\mathbf{x}}$, and setting the result equal to $\mathbf{0}$, as the right-hand side is convex. That is, the minimizer $\hat{\mathbf{x}}$ is such that
\begin{align}
\gamma \left(\hat{\mathbf{x}} - \mathbf{x} \right)^T - \lambda\left(\mathbf{d}_i - \nabla \hat{\mathbf{x}} - \mathbf{b}_i\right)^T\nabla = \mathbf{0}^T,
\end{align}
or equivalently
\begin{align}
\left(\gamma I + \nabla^T\nabla\right)\hat{\mathbf{x}} = \lambda\nabla^T\left(\mathbf{d}_i - \mathbf{b}_i\right) + \gamma\mathbf{x}.
\end{align}
We now attain the solution for \eqref{eq4.5.6} by solving for $\hat{\mathbf{x}}:$
\begin{align}
\hat{\mathbf{x}}_i = \left(\gamma I + \nabla^T\nabla\right)^{-1}\left(\lambda\nabla^T\left(\mathbf{d}_i - \mathbf{b}_i\right) + \gamma\mathbf{x}\right). \label{eq4.5.9}
\end{align}

Next, the algorithm produces $\mathbf{d}_{i+1}$ coordinatewise by
\begin{align}
d_{i+1,j} = \shrink\bigl((\nabla\mathbf{c}_{i+1})_j + b_{i,j}, 1/\lambda\bigr),
\end{align}
where
\begin{align}
\shrink(x,\delta) = \frac{x}{\lvert x \rvert} \cdot \max(\lvert x \rvert - \delta,0).
\end{align}
Once the $n$ iterations of \eqref{eq4.5.5} are carried out, we have our denoised $\hat{\mathbf{x}}$ in the form of \eqref{eq4.5.9}.

The above algorithm was applied to our system with $\gamma = 0.5$, $\lambda = 20,$ and $n = 10.$ The parameter estimates obtained are displayed in \autoref{tab: tv}. These results are comparable to those resulting from Butterworth filtering. The parameter with the lowest input and lowest observed value, $\tau_{01},$ is still estimated with the greatest error. This error suggests parameters with smaller values come with greater uncertainty. To explain this, we note that a smaller parameter value implies the associated transition is observed less frequently, perhaps making estimation schemes more sensitive to noise. 

\begin{table}[!ht] \centering
\caption{Estimated values for $\tau_{ij}$ and \textcolor{blue}{relative error} after adding noise, using total variation}
\begin{tabular}{|c|c|c|c|}
	\cline{2-4}
	\multicolumn{1}{c|}{} & Input ($\times 10^{-3}$) & Observed ($\times 10^{-3}$) & Estimate ($\times 10^{-3}$) \\ \cline{1-4}
	$\tau_{01}$ & 1.00 & 0.76 & 3.45 \textcolor{blue}{$(374.69\%)$} \\ \hline
	$\tau_{02}$ & 6.00 & 5.83 & 4.44 \textcolor{blue}{$(23.79\%)$} \\ \hline
	$\tau_{10}$ & 2.00 & 3.52 & 4.44 \textcolor{blue}{$(26.30\%)$} \\ \hline
	$\tau_{12}$ & 3.00 & 2.24 & 2.86 \textcolor{blue}{$(27.59\%)$} \\ \hline
	$\tau_{20}$ & 4.00 & 5.13 & 4.62 \textcolor{blue}{$(9.86\%)$} \\ \hline
	$\tau_{21}$ & 5.00 & 5.13 & 3.08 \textcolor{blue}{$(39.91\%)$} \\ \hline
\end{tabular}
\label{tab: tv}
\end{table}

As can be seen in \autoref{tab: tv}, the total variation approach tends to create parameter estimates exhibiting less variance than the actual, observed values. That is, parameters with high values are often underestimated, and parameters with low values are often overestimated. This may be a result of a marked shortcoming of TV: the ability to preserve sharp changes in the data. While TV eliminates the spurious oscillations created by noise, it follows state transitions too early. \hyperref[subfig: denoise4]{Figure \ref*{subfig: denoise4}} illustrates this feature.

\begin{figure}[!htb]
\subfloat[Locally Weighted Polynomial Regression]{
\includegraphics[width=0.47\textwidth]{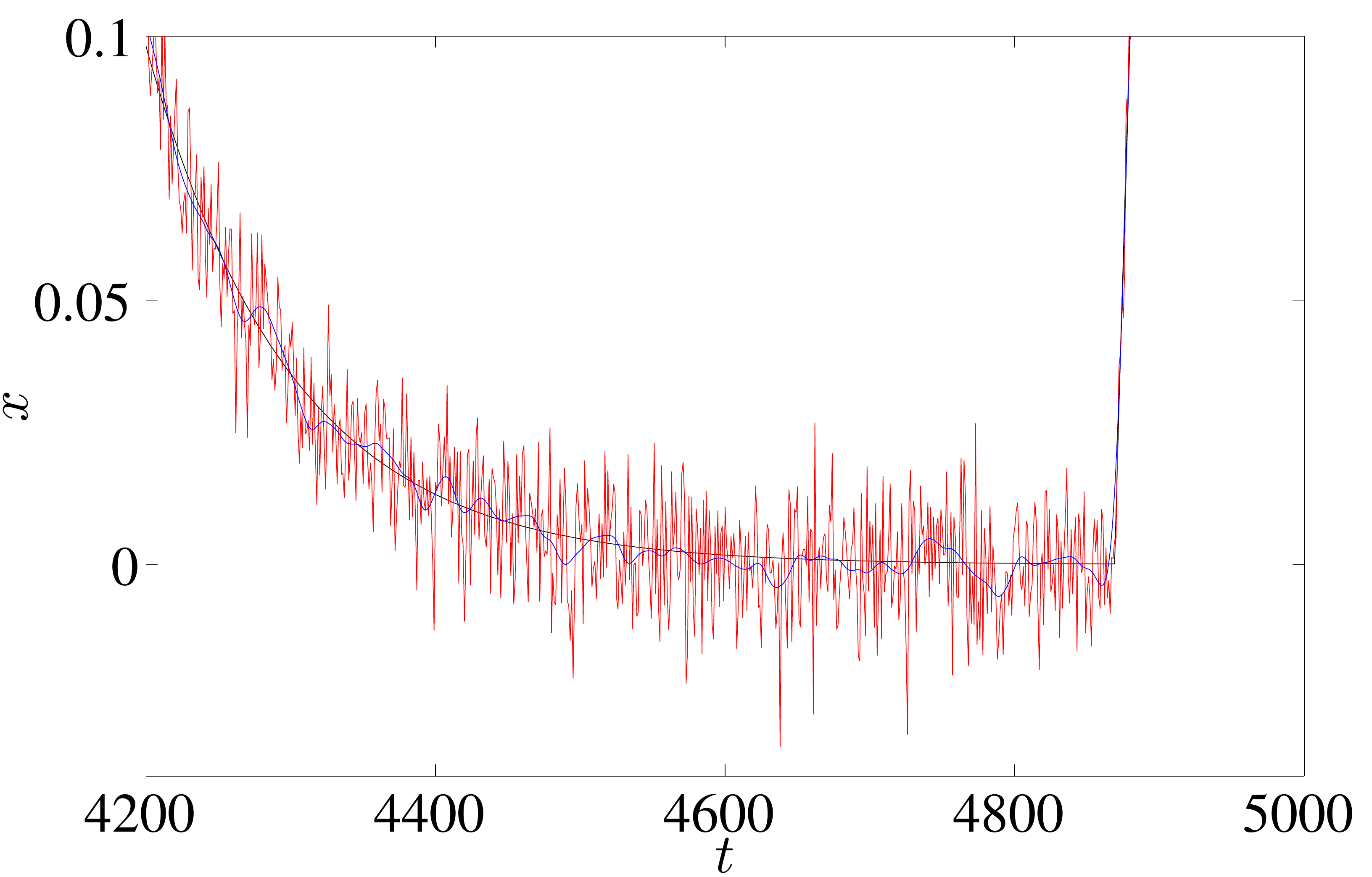}
\label{subfig: denoise1}	
}
\hfill
\subfloat[Wavelet Filtering]{
\includegraphics[width=0.47\textwidth]{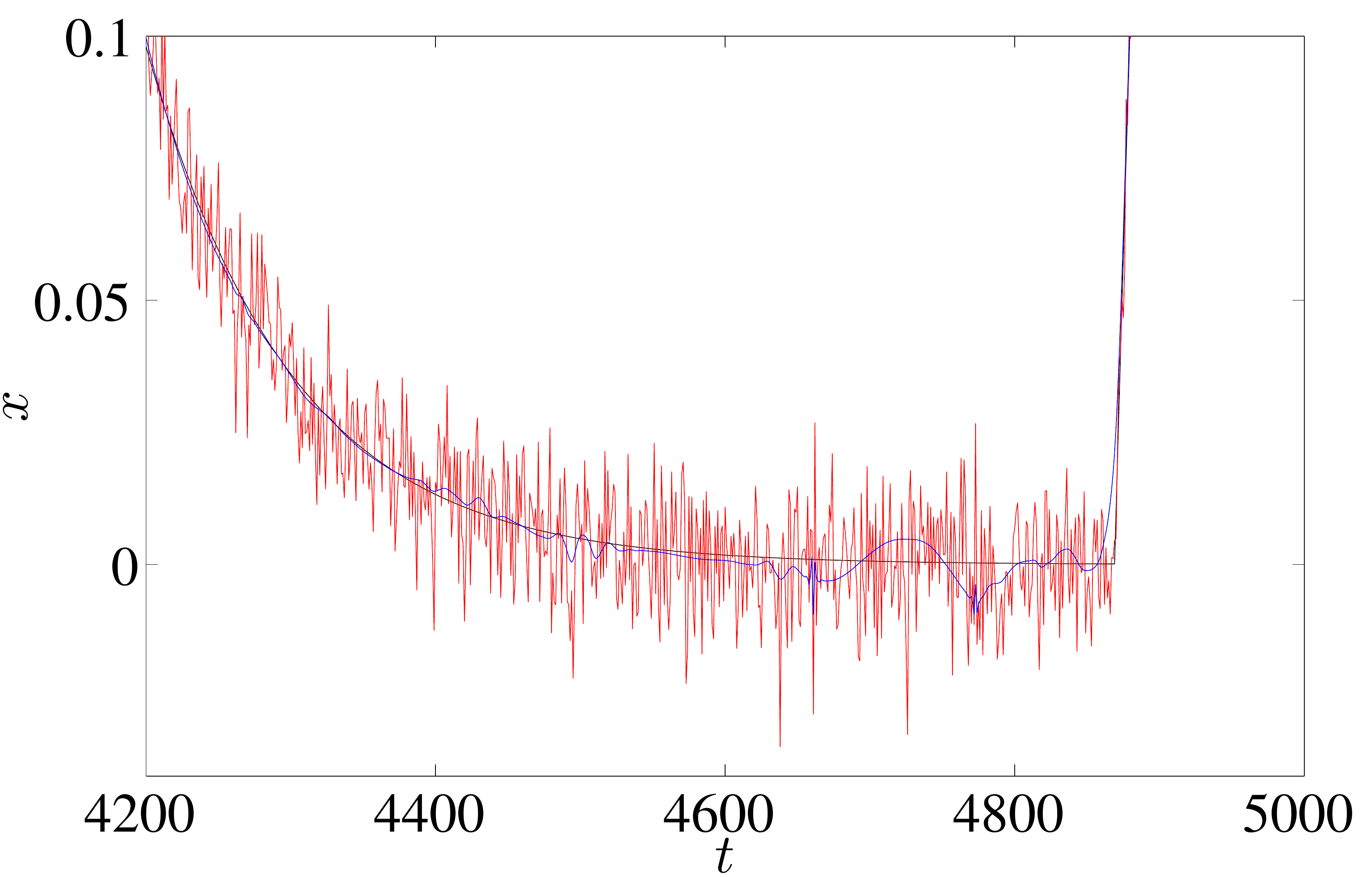}
\label{subfig: denoise2}
}
\hfill
\subfloat[Butterworth Filtering]{
\includegraphics[width=0.47\textwidth]{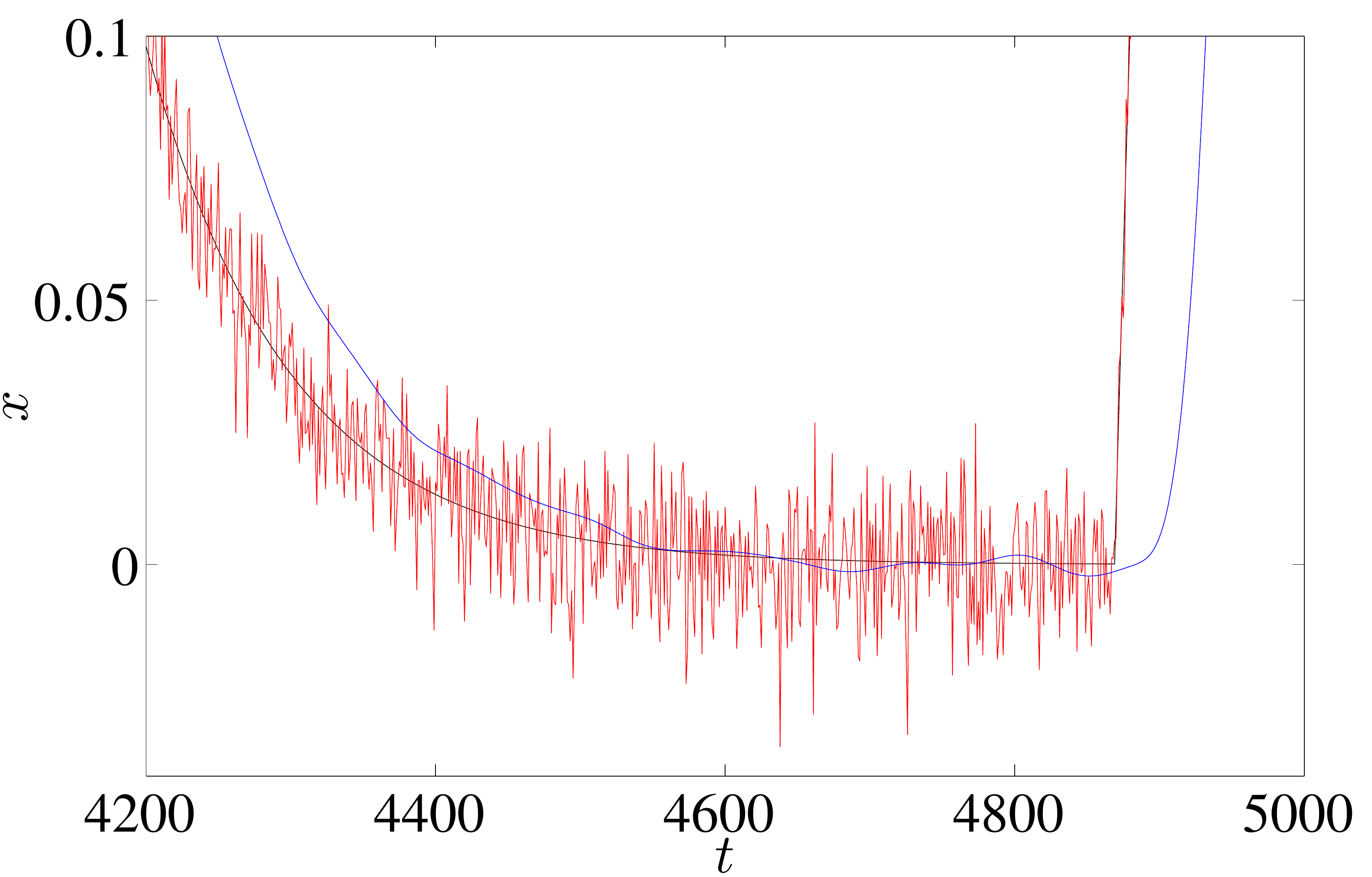}
\label{subfig: denoise3}
}
\hfill
\subfloat[Total Variation]{
\includegraphics[width=0.47\textwidth]{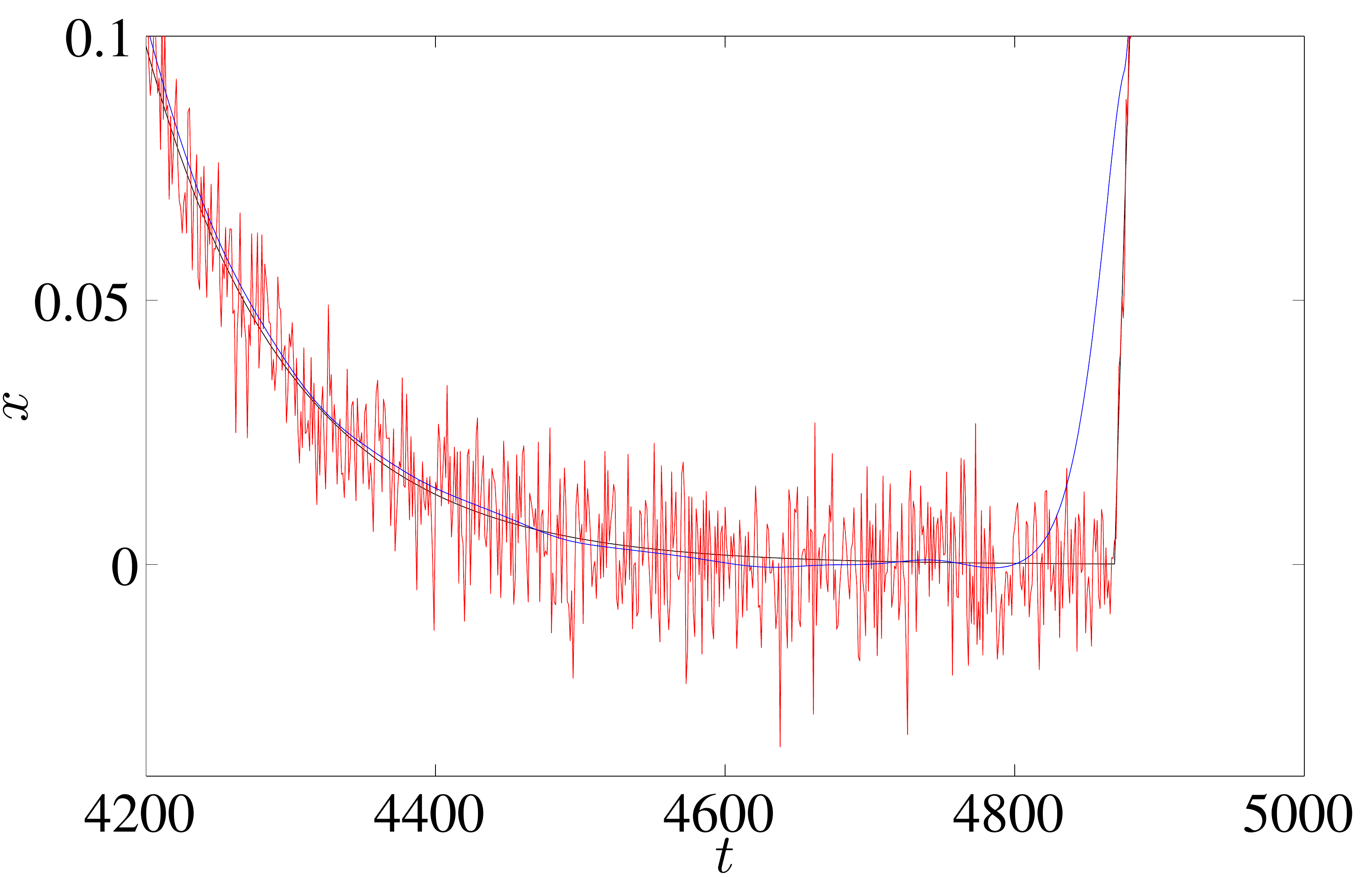}
\label{subfig: denoise4}
}
\caption{Graphical comparison of denoising techniques applied to the data vector $\mathbf{x}$}
\label{fig: denoise}
\end{figure}

\autoref{fig: denoise} also provides a graphical explanation for the deficiency of LWPR and wavelet filtering. The artificial oscillations created by noise are not fully removed (seen in \hyperref[subfig: denoise1]{Figure \ref*{subfig: denoise1}} and \hyperref[subfig: denoise2]{Figure \ref*{subfig: denoise2}}), leading to over-detection of transitions. Butterworth filtering (\hyperref[subfig: denoise3]{Figure \ref*{subfig: denoise3}}) both removes the artifacts of noise and captures the rapid state transition, even if slightly late.

%%%%%%%%%%%%%%%%%%%%%%%%%%%%%%%%%%%%%%%%%%%%%%%%%%%

\section{Conclusion}
When Buridan's ass is modeled as a dynamical system with a discrete-time, discrete-state Markov process, there are several effective techniques of parameter estimation. These techniques include method of moments type estimators, likelihood-based estimators, and state detection.  In higher dimensional extensions of this problem, however, the method of moments approach becomes unmanageable. For $n$ states, there are $n(n-1)$ parameters to estimate, and a method of moments approach requires an equal number of invertible, independent statistics. Moreover, deriving expressions for moments such as mean and variance is made difficult by potentially ill-posed partial differential equations. These problematic PDEs also eliminate the possibility of performing simple, likelihood-based estimation.

The state detector is the one method that generalizes to higher dimensions. Furthermore, state detection is possible in any convex geometry. Therefore, we focus on this method to estimate our parameters in the three-state problem. State detection provides us with the best results we hope to achieve from observing an ideal, noiseless system. Unfortunately, an ideal donkey is not a realistic donkey. Adding noise to the system provides a more realistic problem. State detection alone, however, performs poorly in estimating the parameters, as transitions are over-detected in the presence of noise.  Combining our state detector with denoising techniques relieves much of this over-detection. Of the denoising techniques presented, Butterworth filtering and TV lead to the best parameter estimates.

Further research lies in refining the denoising techniques shown, as well as exploring other denoising methods, to produce better estimates of the parameters.  Any of the estimation schemes presented could also be used to inform Markov chain Monte Carlo simulations. Provided with good initial estimates, such simulations can improve convergence by searching over narrower parameter regimes.

%%%%%%%%%%%%%%%%%%%%%%%%%%%%%%%%%%%%%%%%%%%%%%%%%%%

\section*{Appendices} \addcontentsline{toc}{section}{Appendices}
\renewcommand{\thesubsection}{\Alph{subsection}}

\subsection{An Alternative Perspective}
Instead of a discrete-time Markov process, we can use a continuous-time Poisson process to model the stochastic switching nature of our dynamical systems. This formulation involves parameters dictating the average time observed between transitions, rather than parameters dictating the probability of transitioning in each unit of time.  Nevertheless, the task of estimating these new parameters presents the same challenges associated with measurement error as before. An implementation of this model would again assume knowledge of the donkey's position at finitely many time steps.

In this alternative perspective, we assume that transitions (e.g.~state 0 to state 1) each describe a Poisson process, so that the elapsed times between transitions follow exponential distributions. For instance, in the case of donkey on a line, a transition to state 0 is followed by some random amount of time $dt_{01}$ (before a transition to state 1) sampled from an exponential distribution with mean parameter $\mu_{01}.$ That is, the random variable $dt_{01}$ has density
\begin{align}
f_{01}(t) = \begin{cases} \frac{1}{\mu_{01}} e^{-\frac{1}{\mu_{01}}t} &\text{if}\quad t \geq 0 \\
0 &\text{if}\quad t < 0. \end{cases} \label{eq1}
\end{align}
Waiting times for transitions from state 1 to state 0 are also governed by an exponential distribution, just with a different mean parameter $\mu_{10}.$

Considering donkey in a triangle, there are two possible transitions from each state. Thus, the dynamical system contains six mean parameters to be estimated. Now, transitions away from a state are viewed as two competing Poisson processes. We can consider two $dt$'s for the two possible transitions, and the transition with the smaller $dt$ will occur.

Generalizing this concept and  \eqref{eq1} to any number of states, we say that waiting times for transitions from state $i$ to state $j$ are taken from the exponential distribution
\begin{align}
f_{ij}(t) = \begin{cases} \frac{1}{\mu_{ij}} e^{-\frac{1}{\mu_{ij}}t} &\text{if}\quad t \geq 0 \\
0 &\text{if}\quad t < 0. \end{cases} \label{eq2}
\end{align}
We will assume $0 < \mu_{ij} < \infty$ for all $i$ and $j.$ From \eqref{eq2} we obtain the cumulative distribution function
\begin{align}
F_{ij}(t) &= \int_{-\infty}^t f_{ij}(s) \ ds \\
&= \int_0^t f_{ij}(s) \ ds \\
&= \int_0^t \frac{1}{\mu_{ij}} e^{-\frac{1}{\mu_{ij}}s} \ ds \\
&= -e^{-\frac{1}{\mu_{ij}}s} \Big|_0^t \\
&= 1 - e^{-\frac{1}{\mu_{ij}}t}.
\end{align}
With probability $F_{ij}(t)$ (assuming the donkey's only possible transition is to state $j$), the donkey spends less than $t$ units of time in state $i$ before making a transition to state $j$; that is, $dt_{ij} < t.$ Rather, $dt_{ij} > t$ with probability 
\begin{align}
1 - F_{ij}(t) = e^{-\frac{1}{\mu_{ij}}}. \label{eq8}
\end{align}

This observation leads to the following property of the Poisson process: For any $s,t > 0,$
\begin{align}
\Pr(dt_{ij} > t + s | dt_{ij} > t) &= \frac{\Pr(dt_{ij} > t + s)}{\Pr(dt_{ij} > t)} \\
&= \frac{e^{-\frac{1}{\mu_{ij}}(t+s)}}{e^{-\frac{1}{\mu_{ij}}t}} \\
&= e^{-\frac{1}{\mu_{ij}}s} \\
&= \Pr(dt_{ij} > s).
\end{align}
Thus, the likelihood of a transition occurring in a given length of time is independent of when the previous transition occurred. This property is the continuous analog of the time independence that characterizes the steps of a discrete Markov process.

%%%%%%%%%%%%%%%%%%%%%%%%%%%%%%%%%%%%%%%%%%%%%%%%%%%

\subsubsection{Expected Waiting Times}
As stated before, when more than one type of transition is possible (i.e.~the system has more than two states), we must consider multiple $dt$'s, each one an independent, random sample from a separate exponential distribution. But suppose we wanted to know how long, on average, the donkey spends in a given state. In other words, we wish to determine the expected waiting time for {\it any} transition to occur.

Preserving generality, let there be $n$ states, and let the donkey be in state $i.$ We are interested in how long we wait before observing the next transition: the time elapsed (since the transition to state $i$) will be denoted $dt_i.$ This quantity is given by
\begin{align}
dt_i = \min(dt_{i0},\dots,dt_{i,i-1},dt_{i,i+1},\dots,dt_{i,n-1}).
\end{align}
The probability of having had a transition by time $t$ is
\begin{align}
\Pr(dt_i < t) &= 1 - \Pr(dt_i > t) \\
&= 1 - \Pr(dt_{i0} > t) \cdots \Pr(dt_{i,i-1} > t) \Pr(dt_{i,i+1} > t) \cdots \Pr(dt_{i,n-1} > t) \\
&= 1 - \left(e^{-\frac{1}{\mu_{i0}}t}\right) \cdots \left(e^{-\frac{1}{\mu_{i,i-1}}t}\right) \left(e^{-\frac{1}{\mu_{i,i+1}}t}\right) \cdots \left(e^{-\frac{1}{\mu_{i,n-1}}t}\right) \\
&= 1 - e^{-\left(\frac{1}{\mu_{i0}} + \cdots + \frac{1}{\mu_{i,i-1}} + \frac{1}{\mu_{i,i+1}} + \cdots + \frac{1}{\mu_{i,n-1}}\right)t}.
\end{align}
We recognize this as a cumulative distribution function 
\begin{align}
F_i(t) = 1 - e^{-\left(\sum_{j \neq i} \frac{1}{\mu_{ij}}\right)t}.
\end{align}
Thus, $dt_i$ has density
\begin{align}
f_i(t) &= \frac{d}{dt} F_i \\
&= \sum_{j \neq i} \frac{1}{\mu_{ij}}{e^{-\left(\sum_{j \neq i} \frac{1}{\mu_{ij}}\right)t}}.
\end{align}
This is an exponential distribution with mean parameter 
\begin{align}
\mu_i = \left(\sum_{j \neq i} \frac{1}{\mu_{ij}}\right)^{-1}.
\end{align}
That is, the average waiting time to observe a transition away from state $i$ is $\mu_i$ units of time. Assuming the donkey's state information is detectable, we can empirically observe $\mu_i$ and use such a measurement to estimate the various $\mu_{ij}.$

%%%%%%%%%%%%%%%%%%%%%%%%%%%%%%%%%%%%%%%%%%%%%%%%%%%

\subsubsection{Transition Probabilities}
We next ask what the probability is of observing a particular transition: If the donkey is in state $i$, what is the probability $p_{ij}$ that his next transition is to state $j?$ To calculate this probability, we need only to realize that doing so is equivalent to calculating the probability that $dt_{ij} > dt_{ik}$
for all $k \neq i,j.$ Since $\Pr(dt_{ij} = dt_{ik} \text{ for some } j \neq k) = 0,$ we do not worry about simultaneous transitions. We find
\begin{equation}
\begin{split}
p_{ij} = \Pr(dt_{ij} > dt_{ik} \text{ for all } k \neq i,j) &= \int_0^\infty \big[\Pr(dt_{ij}=t) \Pr(dt_{i0} > t) \cdots \\
&\phantom{=\int_0^\infty \big[}\Pr(dt_{i,i-1} > t) \Pr(dt_{i,i+1} > t) \cdots \\
&\phantom{=\int_0^\infty \big[}\Pr(dt_{i,j-1} > t) \Pr(dt_{i,j+1} > t) \cdots \\
&\phantom{=\int_0^\infty \big[}\Pr(dt_{i,n-1} > t)\big]\  dt. \label{eq22} \raisetag{2\baselineskip}
\end{split}
\end{equation}
Using \eqref{eq8}, the integral in \eqref{eq22} can be evaluated to give
\begin{align}
\begin{split}
p_{ij} &= \int_0^\infty \bigg[\left(\frac{1}{\mu_{ij}}e^{-\frac{1}{\mu_{ij}}t}\right)\left(e^{-\frac{1}{\mu_{i0}}t}\right) \cdots
\left(e^{-\frac{1}{\mu_{i,i-1}}t}\right)\left(e^{-\frac{1}{\mu_{i,i+1}}t}\right) \cdots \\
&\phantom{=\int_0^\infty \bigg[} \left(e^{-\frac{1}{\mu_{i,j-1}}t}\right)\left(e^{-\frac{1}{\mu_{i,j+1}}t}\right) \cdots
\left(e^{-\frac{1}{\mu_{i,n-1}}t}\right) \bigg]
\end{split} \\
&= \frac{1}{\mu_{ij}} \int_0^\infty e^{-\left(\sum_{k \neq i} \frac{1}{\mu_{ik}}\right)t} \ dt \\
&= -\frac{1}{\mu_{ij}} \left(\sum_{k \neq i} \frac{1}{\mu_{ik}}\right)^{-1} e^{-\left(\sum_{k \neq i} \frac{1}{\mu_{ik}}\right)t} \Bigg|_0^\infty \\
&= \frac{\mu_i}{\mu_{ij}}.
\end{align}
That is, we expect 
\begin{align}
p_{ij} = \frac{\mu_i}{\mu_{ij}} \label{eq27}
\end{align}
of the transitions from state $i$ to be to state $j.$ Again, assuming the donkey's state information is detectable, we can measure $p_{ij}$ (very similar to $\tau_{ij}$ in the Markov chain formulation) and use it to estimate $\mu_{ij}.$

%%%%%%%%%%%%%%%%%%%%%%%%%%%%%%%%%%%%%%%%%%%%%%%%%%%

\subsubsection{Parameter Estimation}
We have nearly derived an expression to estimate the various $\mu_{ij}$ parameters. We now simply solve \eqref{eq27} for $\mu_{ij}$ to yield
\begin{align}
\mu_{ij} = \frac{\mu_i}{p_{ij}}.
\end{align}
Recall that $\mu_i$ and $p_{ij}$ are both easily measured given state information, so the empirical $\mu_{ij}$ can be calculated.

%%%%%%%%%%%%%%%%%%%%%%%%%%%%%%%%%%%%%%%%%%%%%%%%%%%

\subsection{Eigenvector of Markov Matrix with Eigenvalue 1} \label{app: eigenvector}
As presented in \ref{Donkey on a Line: The Markov Process} and \ref{Donkey in a Triangle: The Markov Process}, the normalized eigenvector associated with eigenvalue 1 of the $2\times2$
$$\begin{bmatrix}
1-\tau_{01}&\tau_{10} \\
\tau_{01}&1-\tau_{10}
\end{bmatrix}$$
and the $3\times3$
$$\begin{bmatrix}
1-\tau_{01}-\tau_{02}&\tau_{10}&\tau_{20} \\
\tau_{01}&1-\tau_{10}-\tau_{12}&\tau_{21} \\
\tau_{02}&\tau_{12}&1-\tau_{20}-\tau_{21}
\end{bmatrix}$$
Markov matrices are known. In fact, they can be calculated using standard techniques. With increasing size of the Markov matrix, however, these techniques fail by theoretical necessity: Determining the eigenvalues of an arbitrary $n\times n$ matrix requires solving an $n$th degree polynomial, for which a general solution cannot exist for $n\geq 5$ due to the Abel-Ruffini theorem \citep{rotman}.

We are only interested in determining the eigenvector associated with eigenvalue 1. Here we derive a formula to calculate this eigenvector given an $n\times n$ Markov matrix. The following theorem and proof are taken from \cite{eigenvector}. The result is shown for a general column stochastic matrix.

\begin{theorem} \label{th1}
If $A$ is an $n \times n$ column stochastic matrix, then $A$ has an eigenvalue of 1 and associated eigenvector $\mathbf{v}$ given by 
\begin{align}
v_k = \det(M_{k,k}), \label{eq29}
\end{align}
where $M=A-I$ and $M_{k,k}$ is the matrix formed from $M$ by removing its $k$th row and its $k$th column.
\end{theorem}
\begin{proof}
We wish to show
\begin{align}
A\mathbf{v} &= \mathbf{v} \\
\left(A-I\right)\mathbf{v} &= \mathbf{0} \\
M\mathbf{v} &= \mathbf{0},
\end{align}
or equivalently
\begin{align}
\sum_{j = 1}^{n} m_{i,j} v_j &= 0 \\
\sum_{j = 1}^n m_{i,j} \det M_{j,j} &= 0 \label{eq34}
\end{align}
for $i = 1,\dots,n,$ where $m_{i,j}$ is the entry of $M$ in the $i$th row and the $j$th column. We will show \eqref{eq34} only for $i = 1,$ as the remaining $n-1$ cases are analogous.

First notice that the rows of $M$ are linearly dependent (each column sums to 0, so any row is the negative of the sum of all other rows), implying $\det M = 0.$ We also have
\begin{align}
&\sum_{j = 1}^n m_{1,j} \det M_{j,j} =
m_{1,1} \begin{vmatrix}
m_{2,2} & \cdots & m_{2,n} \\
\vdots & \ddots & \vdots \\
m_{n,2} & \cdots & m_{n,n}
\end{vmatrix}
+ \cdots +
m_{1,n} \begin{vmatrix}
m_{1,1} & \cdots & m_{1,n-1} \\
\vdots & \ddots & \vdots \\
m_{n-1,1} & \cdots & m_{n-1,n-1}
\end{vmatrix} \\
\begin{split}
&= m_{1,1} \begin{vmatrix}
m_{2,2} & \cdots & m_{2,n} \\
\vdots & \ddots & \vdots \\
m_{n,2} & \cdots & m_{n,n}
\end{vmatrix}
+ m_{1,2}\begin{vmatrix}
-\sum_{j=2}^n m_{j,1} & -\sum_{j=2}^n m_{j,3} & \cdots & -\sum_{j=2}^n m_{j,n}  \\
m_{3,1} & m_{3,3} & \cdots & m_{3,n} \\
\vdots & \vdots & \ddots & \vdots \\
m_{n,1} & m_{n,3} & \cdots & m_{n,n}
\end{vmatrix} \\
&\phantom{=}+ \cdots +
m_{1,n} \begin{vmatrix}
-\sum_{j=2}^n m_{j,1} & -\sum_{j=2}^n m_{j,2} & \cdots & -\sum_{j=2}^n m_{j,n-1} \\
m_{2,1} & m_{2,2} & \cdots & m_{2,n-1} \\
\vdots & \vdots & \ddots & \vdots \\
m_{n-1,1} & m_{n-1,2} & \cdots & m_{n-1,n-1}
\end{vmatrix},
\end{split}
\end{align}
where, in the last $n-1$ terms of the sum, we have expressed the first row of the matrix as the negative of the sum of the last $n-1$ rows (with the appropriate columns omitted) of $M.$ In computing the determinants of these matrices, we now remove the negative from the first row and split the sums:
\begin{align}
\begin{split}
\sum_{j = 1}^n m_{1,j} \det M_{j,j} &= 
m_{1,1} \begin{vmatrix}
m_{2,2} & \cdots & m_{2,n} \\
\vdots & \ddots & \vdots \\
m_{n,2} & \cdots & m_{n,n}
\end{vmatrix}
- m_{1,2} \sum_{j=2}^n \begin{vmatrix}
m_{j,1} & m_{j,3} & \cdots & m_{j,n}  \\
m_{3,1} & m_{3,3} & \cdots & m_{3,n} \\
\vdots & \vdots & \ddots & \vdots \\
m_{n,1} & m_{n,3} & \cdots & m_{n,n}
\end{vmatrix} \\
&\phantom{=} - \cdots -
m_{1,n} \sum_{j=2}^n \begin{vmatrix}
m_{j,1} & m_{j,2} & \cdots & m_{j,n-1} \\
m_{2,1} & m_{2,2} & \cdots & m_{2,n-1} \\
\vdots & \vdots & \ddots & \vdots \\
m_{n-1,1} & m_{n-1,2} & \cdots & m_{n-1,n-1}
\end{vmatrix}.
\end{split} \label{eq: detsums}
\end{align}
In each of the $n-1$ sums on the right-hand side of \eqref{eq: detsums}, only one $j$ yields a first row that is not identical to one of the other rows (in particular, for the sum following $m_{1,k},$ this $j$ equals $k$). As all other $j$ yield a determinant of 0, each sum reduces to a single term:
\begin{align}
\begin{split}
\sum_{j = 1}^n m_{1,j} \det M_{j,j} &= 
m_{1,1} \begin{vmatrix}
m_{2,2} & \cdots & m_{2,n} \\
\vdots & \ddots & \vdots \\
m_{n,2} & \cdots & m_{n,n}
\end{vmatrix}
- m_{1,2} \begin{vmatrix}
m_{2,1} & m_{2,3} & \cdots & m_{2,n}  \\
m_{3,1} & m_{3,3} & \cdots & m_{3,n} \\
\vdots & \vdots & \ddots & \vdots \\
m_{n,1} & m_{n,3} & \cdots & m_{n,n}
\end{vmatrix} \\
&\phantom{=} - \cdots -
m_{1,n} \begin{vmatrix}
m_{n,1} & m_{n,2} & \cdots & m_{n,n-1} \\
m_{2,1} & m_{2,2} & \cdots & m_{2,n-1} \\
\vdots & \vdots & \ddots & \vdots \\
m_{n-1,1} & m_{n-1,2} & \cdots & m_{n-1,n-1}
\end{vmatrix}.
\end{split} \label{eq38}
\end{align}
We now shift down the first row of the last $n-2$ matrices in \eqref{eq38} so that rows are in order (for the $m_{1,k}$ term, this requires $k-2$ transpositions, or ``flips," of rows). Furthermore, each transposition changes the sign of the determinant, so we have
\begin{align}
\begin{split}
\sum_{j = 1}^n m_{1,j} \det M_{j,j} &= 
m_{1,1} \begin{vmatrix}
m_{2,2} & \cdots & m_{2,n} \\
\vdots & \ddots & \vdots \\
m_{n,2} & \cdots & m_{n,n}
\end{vmatrix}
+ (-1)^{1} m_{1,2} \begin{vmatrix}
m_{2,1} & m_{2,3} & \cdots & m_{2,n}  \\
m_{3,1} & m_{3,3} & \cdots & m_{3,n} \\
\vdots & \vdots & \ddots & \vdots \\
m_{n,1} & m_{n,3} & \cdots & m_{n,n}
\end{vmatrix} \\
&\phantom{=} + \cdots + (-1)^{n-1}
m_{1,n} \begin{vmatrix}
m_{2,1} & m_{2,2} & \cdots & m_{2,n-1} \\
\vdots & \vdots & \ddots & \vdots \\
m_{n-1,1} & m_{n-1,2} & \cdots & m_{n-1,n-1} \\
m_{n,1} & m_{n,2} & \cdots & m_{n,n-1}
\end{vmatrix}
\end{split} \label{eq39} \\
\begin{split} &=
(-1)^{1+1} m_{1,1} \begin{vmatrix}
m_{2,2} & \cdots & m_{2,n} \\
\vdots & \ddots & \vdots \\
m_{n,2} & \cdots & m_{n,n}
\end{vmatrix} \\
&\phantom{=}+ (-1)^{1+2} m_{1,2} \begin{vmatrix}
m_{2,1} & m_{2,3} & \cdots & m_{2,n}  \\
m_{3,1} & m_{3,3} & \cdots & m_{3,n} \\
\vdots & \vdots & \ddots & \vdots \\
m_{n,1} & m_{n,3} & \cdots & m_{n,n}
\end{vmatrix} \\
&\phantom{=} + \cdots + (-1)^{1+n}
m_{1,n} \begin{vmatrix}
m_{2,1} & m_{2,2} & \cdots & m_{2,n-1} \\
\vdots & \vdots & \ddots & \vdots \\
m_{n-1,1} & m_{n-1,2} & \cdots & m_{n-1,n-1} \\
m_{n,1} & m_{n,2} & \cdots & m_{n,n-1}
\end{vmatrix}
\end{split} \\
&= \sum_{j = 1}^n (-1)^{1+j} m_{1,j} \det M_{1,j}. \label{eq41}
\end{align}
We recognize \eqref{eq41} as Laplace's formula by cofactor expansion for the determinant \citep{treil}, so we have
\begin{align}
\sum_{j = 1}^n m_{1,j} \det M_{1,j} &= \sum_{j = 1}^n (-1)^{1+j} m_{1,j} \det M_{1,j} \\
&= \det M \label{eq42} \\
&= 0.
\end{align}
\end{proof}

%%%%%%%%%%%%%%%%%%%%%%%%%%%%%%%%%%%%%%%%%%%%%%%%%%%

\subsection{A Conjecture Regarding Eigenvectors} \label{app: joseph}
\autoref{th1} provides a way of computing the eigenvector corresponding to the eigenvalue 1 of the $n\times n$ Markov matrix
\begin{align}
A = \scalebox{0.9}{$\begin{bmatrix}
1 - \tau_{01} -  \cdots - \tau_{0,n-1} & \tau_{10} & \cdots & \tau_{n-1,0} \\
\tau_{01} & 1 - \tau_{10} - \tau_{12} - \cdots - \tau_{1,n-1} & \cdots & \tau_{n-1,1} \\
\vdots & \vdots & \ddots & \vdots \\
\tau_{0,n-1} & \tau_{1,n-1} & \cdots & 1 - \tau_{n-1,0} - \cdots - \tau_{n-1,n-2}
\end{bmatrix}$}. \label{eq44}
\end{align}
Nevertheless, the computation \eqref{eq29} has increasing complexity with increasing $n.$ For example, using MATLAB\textsuperscript{\textregistered}, we were able to compute the desired eigenvector only for $n \leq 6.$ For $n > 6,$ the computation proved too complex to execute. To gain a sense of this complexity, we examined the growth of the number of summands in a single application of \eqref{eq29}.

For example, in the $2\times 2$ case, $A$ has the normalized eigenvector
$$\mathbf{v} = \begin{bmatrix}
\dfrac{\tau_{10}}{\tau_{01}+\tau_{10}} \\[0.4cm]
\dfrac{\tau_{01}}{\tau_{01}+\tau_{10}}
\end{bmatrix}.$$
In the $3\times 3$ case,
$$\mathbf{v} = \begin{bmatrix}
\dfrac{\tau_{21}\tau_{10} + \tau_{12}\tau_{20} + \tau_{10}\tau_{20}}{\tau_{21}\tau_{10} + \tau_{12}\tau_{20} + \tau_{10}\tau_{20} + \tau_{20}\tau_{01} + \tau_{02}\tau_{21} + \tau_{01}\tau_{21} + \tau_{10}\tau_{02} + \tau_{01}\tau_{12} + \tau_{02}\tau_{12}} \\[0.4cm]
\dfrac{\tau_{20}\tau_{01} + \tau_{02}\tau_{21} + \tau_{01}\tau_{21}}{\tau_{21}\tau_{10} + \tau_{12}\tau_{20} + \tau_{10}\tau_{20} + \tau_{20}\tau_{01} + \tau_{02}\tau_{21} + \tau_{01}\tau_{21} + \tau_{10}\tau_{02} + \tau_{01}\tau_{12} + \tau_{02}\tau_{12}} \\[0.4cm]
\dfrac{\tau_{10}\tau_{02} + \tau_{01}\tau_{12} + \tau_{02}\tau_{12}}{\tau_{21}\tau_{10} + \tau_{12}\tau_{20} + \tau_{10}\tau_{20} + \tau_{20}\tau_{01} + \tau_{02}\tau_{21} + \tau_{01}\tau_{21} + \tau_{10}\tau_{02} + \tau_{01}\tau_{12} + \tau_{02}\tau_{12}}
\end{bmatrix}. $$
With four states, the eigenvector (not normalized) becomes
$$\mathbf{v} = \scalebox{0.94}{$\begin{bmatrix}
\tau_{10}\tau_{20}\tau_{30} + \tau_{10}\tau_{20}\tau_{31} + \tau_{10}\tau_{21}\tau_{30} + \tau_{10}\tau_{20}\tau_{32} + \tau_{10}\tau_{21}\tau_{31} + \tau_{12}\tau_{20}\tau_{30} + \tau_{10}\tau_{21}\tau_{32} + \tau_{10}\tau_{23}\tau_{30} \\ + \tau_{12}\tau_{20}\tau_{31} + \tau_{13}\tau_{20}\tau_{30} + \tau_{10}\tau_{23}\tau_{31} + \tau_{12}\tau_{20}\tau_{32} + \tau_{13}\tau_{21}\tau_{30} + \tau_{12}\tau_{23}\tau_{30} + \tau_{13}\tau_{20}\tau_{32} + \tau_{13}\tau_{23}\tau_{30} \\[0.3cm]
\tau_{01}\tau_{20}\tau_{30 } + \tau_{01}\tau_{20}\tau_{31 } + \tau_{01}\tau_{21}\tau_{30 } + \tau_{01}\tau_{20}\tau_{32 } + \tau_{01}\tau_{21}\tau_{31 } + \tau_{02}\tau_{21}\tau_{30 } + \tau_{01}\tau_{21}\tau_{32 } + \tau_{01}\tau_{23}\tau_{30 } \\ + \tau_{02}\tau_{21}\tau_{31 } + \tau_{03}\tau_{20}\tau_{31 } + \tau_{01}\tau_{23}\tau_{31 } + \tau_{02}\tau_{21}\tau_{32 } + \tau_{03}\tau_{21}\tau_{31 } + \tau_{02}\tau_{23}\tau_{31 } + \tau_{03}\tau_{21}\tau_{32 } + \tau_{03}\tau_{23}\tau_{31} \\[0.3cm]
\tau_{02}\tau_{10}\tau_{30 } + \tau_{01}\tau_{12}\tau_{30 } + \tau_{02}\tau_{10}\tau_{31 } + \tau_{01}\tau_{12}\tau_{31 } + \tau_{02}\tau_{10}\tau_{32 } + \tau_{02}\tau_{12}\tau_{30 } + \tau_{01}\tau_{12}\tau_{32 } + \tau_{02}\tau_{12}\tau_{31 } \\ + \tau_{02}\tau_{13}\tau_{30 } + \tau_{03}\tau_{10}\tau_{32 } + \tau_{01}\tau_{13}\tau_{32 } + \tau_{02}\tau_{12}\tau_{32 } + \tau_{03}\tau_{12}\tau_{31 } + \tau_{02}\tau_{13}\tau_{32 } + \tau_{03}\tau_{12}\tau_{32 } + \tau_{03}\tau_{13}\tau_{32} \\[0.3cm]
\tau_{03}\tau_{10}\tau_{20 } + \tau_{01}\tau_{13}\tau_{20 } + \tau_{03}\tau_{10}\tau_{21 } + \tau_{01}\tau_{13}\tau_{21 } + \tau_{02}\tau_{10}\tau_{23 } + \tau_{03}\tau_{12}\tau_{20 } + \tau_{01}\tau_{12}\tau_{23 } + \tau_{02}\tau_{13}\tau_{21 } \\ + \tau_{03}\tau_{10}\tau_{23 } + \tau_{03}\tau_{13}\tau_{20 } + \tau_{01}\tau_{13}\tau_{23 } + \tau_{02}\tau_{12}\tau_{23 } + \tau_{03}\tau_{13}\tau_{21 } + \tau_{02}\tau_{13}\tau_{23 } + \tau_{03}\tau_{12}\tau_{23 } + \tau_{03}\tau_{13}\tau_{23}
\end{bmatrix}$}.$$

It is easy to see from \eqref{eq29} that for a given $n,$ each entry of the eigenvector $\mathbf{v}$ will contain the same number of summands before normalizing. Counting the number of monomials in each numerator after normalizing, we see a possible pattern, illustrated in \autoref{tab: joseph}.

\begin{table}[ht]  \centering
\caption{Number of monomials in a coordinate of the eigenvector corresponding to eigenvalue 1 of Markov matrix}
\begin{tabular}{|c|c|l|} \cline{1-2}
\textbf{Number of States} & \textbf{Terms in Numerators} & \multicolumn{1}{|}{}\\ \hline
2 & 1 & $=2^0$ \\ \hline
3 & 3 & $=3^1$ \\ \hline
4 & 16 & $=4^2$ \\ \hline
5 & 125 & $=5^3$ \\ \hline
6 & 1296 & $=6^4$ \\ \hline
\end{tabular} \label{tab: joseph}
\end{table}

\noindent We are lead to the following conjecture.
\begin{conjecture} \label{conj: joseph}
If $A$ is the $n\times n$ matrix  given by \eqref{eq44}, then each coordinate of the eigenvector $\mathbf{v}$ given by \eqref{eq29}, when expressed in terms of the various $\tau_{ij},$ is the sum of $n^{n-2}$ monomials, each with coefficient 1.
\end{conjecture}
If this conjecture is true, a derivation may proceed as follows. We let $\mathcal{R}$ be the number of monomials in a coordinate of $\mathbf{v},$ the number we conjecture is equal to $n^{n-2}.$ Let $M = A - I$ as in \hyperref[app: eigenvector]{Appendix \ref*{app: eigenvector}}. As prescribed by \autoref{th1}, to compute the $k$th entry of $\mathbf{v},$ we must calculate $\det M_{k,k}.$ To do so, we will use the Leibniz formula \citep{treil}, presented here for a general matrix $B$:
\begin{align}
\det B = \sum_{\sigma \in S_n} \sgn \sigma \prod_{i=1}^n b_{i,\sigma(i)}, \label{eq45}
\end{align}
where $S_n$ is the symmetric group on $n$ elements, and $b_{i,j}$ denotes an entry in $B.$

In applying this formula, for example, to
$$M_{n,n} = \scalebox{0.91}{$\begin{bmatrix}
- \tau_{01} -  \cdots - \tau_{0,n-1} & \tau_{10} & \cdots & \tau_{n-2,0} \\
\tau_{01} & - \tau_{10} - \tau_{12} - \cdots - \tau_{1,n-1} & \cdots & \tau_{n-2,1} \\
\vdots & \vdots & \ddots & \vdots \\
\tau_{0,n-2} & \tau_{1,n-2} & \cdots & - \tau_{n-2,0} - \cdots - \tau_{n-2,n-3} - \tau_{n-2,n-1}
\end{bmatrix}$},$$
we notice that
\begin{enumerate}
\item Each term in the sum of \eqref{eq45} is a product of $n-1$ entries of $M_{n,n},$ exactly one from each row and column.
\item In a particular summand of \eqref{eq45}, the sign of each monomial produced is the same and is determined by two things. First, each diagonal entry in the product forming the monomial contributes a power of $-1.$ Second, the signum of the permutation involved could contribute an additional factor of $-1.$
\item Diagonal entries being involved in a particular summand of \eqref{eq45} is equivalent to the permutation involved having a fixed point.
\item The number of diagonal entries involved in a particular summand of \eqref{eq45} determines how many monomials the summand contains. In particular, if $f(\sigma)$ denotes the number of fixed points of the permutation $\sigma,$ then the associated summand will contain $(n-1)^{f(\sigma)}$ monomials.
\end{enumerate}

The assumption we now make, albeit possibly incorrect, is that monomials of opposite sign will cancel to leave only monomials with identical sign. In this case, the four observations above imply
\begin{align}
\mathcal{R} = (-1)^{n-1}\sum_{\sigma \in S_n} (\sgn \sigma) (-1)^{f(\sigma)}(n-1)^{f(\sigma)}. \label{eq: final}
\end{align}
The factor of $(-1)^{n-1}$ in \eqref{eq: final} ensures $\mathcal{R}$ is always positive. If $n$ is even, then the monomials comprising $\mathbf{v},$ as determined by \eqref{eq29}, have negative coefficients. This fact is seen by noting that the majority of monomials come from the summand associated with the identity permutation---the summand in \eqref{eq45} that multiplies all diagonal elements. Since the identity permutation has positive signum, the sign of the monomials in this summand is entirely determined by $n$. If $n$ is even, then $n-1$ is odd, so the sign of these monomials is $(-1)^{n-1},$ as there are $n-1$ diagonal entries.

Provided our assumption is correct, proving \hyperref[conj: joseph]{Conjecture \ref*{conj: joseph}} is now reduced to showing the right-hand side of \eqref{eq: final} is equal to $n^{n-2}.$ The sum in \eqref{eq: final} may be easier to manipulate if permutations were grouped by cycle structure.

%%%%%%%%%%%%%%%%%%%%%%%%%%%%%%%%%%%%%%%%%%%%%%%%%%%

\subsection{The Growth of Cumulative Power} \label{app: power}
In \cite{keaton}, the author shows that, for finite sums of sines and cosines, cumulative power \eqref{eq: cumpow} is $\Ot$. Here we prove the same growth relationship for the cumulative power \eqref{eq: cumpow2} of the donkey's position, $x$, in the one-dimensional case.
\begin{theorem}
If $x$ obeys \eqref{eq: state0} and \eqref{eq: state1} and switches between the two states by the Markov process \eqref{eq: markov}, then its cumulative power \eqref{eq: cumpow2} is $\Ot$.
\end{theorem}
\begin{proof}
Notice from \eqref{eq: power1} and \eqref{eq: power2} that
\begin{align}
0 < \bigl(x''(s)\bigr)^2 < v^4, \label{eq: statement1}
\end{align}
since
\begin{align}
0 < x(s) < 1 \label{eq: statement2}
\end{align}
for all time $s,$ under the assumption $0 < x(0) < 1$. Thus, we have
\begin{align}
F(t) &= \int_0^{t_1} \bigl(x''(s)\bigr)^2\ ds + \int_{t_1}^{t_2} \bigl(x''(s)\bigr)^2\ ds + \cdots + \int_{t_{n-1}}^{t_n} \bigl(x''(s)\bigr)^2\ ds + \int_{t_{n}}^{t} \bigl(x''(s)\bigr)^2\ ds \\
&< \int_0^{t_1} v^4\ ds + \int_{t_1}^{t_2} v^4\ ds + \cdots + \int_{t_{n-1}}^{t_n} v^4\ ds + \int_{t_{n}}^{t} v^4\ ds \\
&= v^4\bigl[(t_1 - 0) + (t_2 - t_1) + \cdots + (t_n - t_{n-1}) + (t - t_n)\bigr] \\
&= v^4t, \label{eq: proof}
\end{align}
for all time $t.$ Since $v^4$ is constant and cumulative power $F$ is always positive, \eqref{eq: proof} shows $F \in \Ot$.
\end{proof}

%%%%%%%%%%%%%%%%%%%%%%%%%%%%%%%%%%%%%%%%%%%%%%%%%%%

\section*{Acknowledgments} \addcontentsline{toc}{section}{Acknowledgments}
This work was completed as part of the Summer Undergraduate Research Institute in Experimental Mathematics (SURIEM) at the Lyman Briggs College of Michigan State University. We thank the sponsorship of the National Security Agency and the National Science Foundation in funding this REU program. We thank Professor Daniel P.~Dougherty for his mentorship, as well as Joseph E.~Roth for his assistance. We also wish to credit Mr.~Roth for his observations leading to \hyperref[conj: joseph]{Conjecture \ref*{conj: joseph}}.

%%%%%%%%%%%%%%%%%%%%%%%%%%%%%%%%%%%%%%%%%%%%%%%%%%%

%\nocite{*}
%\bibliography{Final_Report.bib} 

\end{document}